\pdfoutput=1
\documentclass[10pt,a4paper]{amsart} 
\newif\ifpersonal
\newif\ifpersonalsection
\personaltrue
\personalsectiontrue
\usepackage{amsmath,amsthm,amssymb,mathrsfs,mathtools,bm,centernot} 
\usepackage{microtype,lmodern} 
\usepackage{enumitem,comment,braket,xspace,tikz-cd,adjustbox} 
\usepackage[utf8]{inputenc} 
\usepackage[T1]{fontenc} 
\usepackage[centering,vscale=0.7,hscale=0.7]{geometry}
\usepackage[hidelinks,linktoc=all]{hyperref}
\usepackage[capitalize]{cleveref}
\allowdisplaybreaks
\usepackage{graphicx}
\usepackage{tikz-cd}
\graphicspath{ {./images/} }
\numberwithin{equation}{section}
\theoremstyle{plain}
\newtheorem{theorem}[equation]{Theorem}
 
\newtheorem{lemma}[equation]{Lemma}
\newtheorem*{lemma*}{Lemma}

\newtheorem{claim}[equation]{Claim}
\newtheorem*{claim*}{Claim}
\newtheorem{conjecture}[equation]{Conjecture}
\newtheorem{proposition}[equation]{Proposition}
\newtheorem*{proposition*}{Proposition}
\newtheorem{corollary}[equation]{Corollary}
\theoremstyle{definition}
\newtheorem{definition}[equation]{Definition}
\newtheorem{definition-theorem}[equation]{Definition-Theorem}
\newtheorem{definition-lemma}[equation]{Definition-Lemma}
\newtheorem{construction}[equation]{Construction}
\newtheorem{assumption}[equation]{Assumption}
\newtheorem{assumptions}[equation]{Assumptions}

\newtheorem{notation}[equation]{Notation}
\theoremstyle{remark}

\newtheorem{remark}[equation]{Remark}
\newtheorem{remarks}[equation]{Remarks}

\newcommand{\uB}{\underline{B}}

\newcommand{\ocY}{\widebar{\cY}}
\newcommand{\hcY}{\widehat{\cY}}

\DeclareFontFamily{U}{BOONDOX-calo}{\skewchar\font=45 }
\DeclareFontShape{U}{BOONDOX-calo}{m}{n}{<-> s*[1.05] BOONDOX-r-calo}{}
\DeclareFontShape{U}{BOONDOX-calo}{b}{n}{<-> s*[1.05] BOONDOX-b-calo}{}
\DeclareMathAlphabet{\mathcalboondox}{U}{BOONDOX-calo}{m}{n}

\newcommand{\ADM}{\operatorname{ADM}}

\newcommand{\an}{\mathrm{an}}

\newcommand{\gp}{\mathrm{gp}}

\newcommand{\mss}{\operatorname{ss}}
\newcommand{\msss}{\operatorname{sss}}

\newcommand{\sm}{\mathrm{sm}}

\newcommand{\trop}{\mathrm{trop}}

\DeclareMathOperator{\Aut}{Aut}

\DeclareMathOperator{\bogus}{Bogus}

\DeclareMathOperator{\Cox}{Cox}

\DeclareMathOperator{\Eff}{Eff}

\DeclareMathOperator{\Hom}{Hom}

\DeclareMathOperator{\Ample}{Ample}
\DeclareMathOperator{\MovFan}{MovFan}

\DeclareMathOperator{\NE}{NE}

\DeclareMathOperator{\Nef}{Nef}
\DeclareMathOperator{\Pic}{Pic}

\DeclareMathOperator{\Proj}{Proj}

\DeclareMathOperator{\Sk}{Sk}

\DeclareMathOperator{\Spec}{Spec}
\DeclareMathOperator{\uSpec}{\underline{\Spec}}
\DeclareMathOperator{\Sing}{Sing}

\DeclareMathOperator{\tor}{tor}

\DeclareMathOperator{\TV}{TV}
\DeclareMathOperator{\hTV}{\widehat{\TV}}

\DeclareMathOperator{\ex}{ex}

\ifpersonal
\newcommand{\todo}[1]{\textcolor{red}{(Todo: #1)}}
\newcommand{\personal}[1]{\textcolor[rgb]{0,0,1}{(Personal: #1)}}
\newcommand{\discussion}[1]{\textcolor{violet}{(Discussion: #1)}}
\else
\newcommand{\todo}[1]{\ignorespaces}
\newcommand{\personal}[1]{\ignorespaces}
\newcommand{\discussion}[1]{\ignorespaces}
\fi

\newcommand{\bbA}{\mathbb A}

\newcommand{\bbC}{\mathbb C}

\newcommand{\bbE}{\mathbb E}

\newcommand{\bbG}{\mathbb G}

\newcommand{\bbN}{\mathbb N}
\newcommand{\bbP}{\mathbb P}

\newcommand{\bbQ}{\mathbb Q}
\newcommand{\bbR}{\mathbb R}
\newcommand{\bbT}{\mathbb T}

\newcommand{\bbU}{\mathbb U}
\newcommand{\bbV}{\mathbb V}
\newcommand{\VVert}{{\bbV}_{\operatorname{ert}}}
\newcommand{\obbV}{\widebar{\mathbb V}}

\newcommand{\bbY}{\mathbb Y}
\newcommand{\hbbY}{\widehat{\bbY}}

\newcommand{\bbZ}{\mathbb Z}
\newcommand{\bbX}{\mathbb X}
\newcommand{\obbX}{\overline{\bbX}}

\newcommand{\cA}{\mathcal A}

\newcommand{\cD}{\mathcal D}

\newcommand{\cE}{\mathcal E}

\newcommand{\cL}{\mathcal L}

\newcommand{\cO}{\mathcal O}

\newcommand{\cT}{\mathcal T}

\newcommand{\cX}{\mathcal X}
\newcommand{\cY}{\mathcal Y}

\makeatletter
\let\save@mathaccent\mathaccent
\newcommand*\if@single[3]{%
	\setbox0\hbox{${\mathaccent"0362{#1}}^H$}%
	\setbox2\hbox{${\mathaccent"0362{\kern0pt#1}}^H$}%
	\ifdim\ht0=\ht2 #3\else #2\fi
}
\newcommand*\rel@kern[1]{\kern#1\dimexpr\macc@kerna}
\newcommand*\widebar[1]{\@ifnextchar^{{\wide@bar{#1}{0}}}{\wide@bar{#1}{1}}}
\newcommand*\wide@bar[2]{\if@single{#1}{\wide@bar@{#1}{#2}{1}}{\wide@bar@{#1}{#2}{2}}}
\newcommand*\wide@bar@[3]{%
	\begingroup
	\def\mathaccent##1##2{%
		\let\mathaccent\save@mathaccent
		\if#32 \let\macc@nucleus\first@char \fi
		\setbox\z@\hbox{$\macc@style{\macc@nucleus}_{}$}%
		\setbox\tw@\hbox{$\macc@style{\macc@nucleus}{}_{}$}%
		\dimen@\wd\tw@
		\advance\dimen@-\wd\z@
		\divide\dimen@ 3
		\@tempdima\wd\tw@
		\advance\@tempdima-\scriptspace
		\divide\@tempdima 10
		\advance\dimen@-\@tempdima
		\ifdim\dimen@>\z@ \dimen@0pt\fi
		\rel@kern{0.6}\kern-\dimen@
		\if#31
		\overline{\rel@kern{-0.6}\kern\dimen@\macc@nucleus\rel@kern{0.4}\kern\dimen@}%
		\advance\dimen@0.4\dimexpr\macc@kerna
		\let\final@kern#2%
		\ifdim\dimen@<\z@ \let\final@kern1\fi
		\if\final@kern1 \kern-\dimen@\fi
		\else
		\overline{\rel@kern{-0.6}\kern\dimen@#1}%
		\fi
	}%
	\macc@depth\@ne
	\let\math@bgroup\@empty \let\math@egroup\macc@set@skewchar
	\mathsurround\z@ \frozen@everymath{\mathgroup\macc@group\relax}%
	\macc@set@skewchar\relax
	\let\mathaccentV\macc@nested@a
	\if#31
	\macc@nested@a\relax111{#1}%
	\else
	\def\gobble@till@marker##1\endmarker{}%
	\futurelet\first@char\gobble@till@marker#1\endmarker
	\ifcat\noexpand\first@char A\else
	\def\first@char{}%
	\fi
	\macc@nested@a\relax111{\first@char}%
	\fi
	\endgroup
}
\makeatother

\newcommand{\oT}{\widebar T}

\newcommand{\oV}{\widebar V}

\newcommand{\oX}{\widebar X}

\newcommand{\tV}{\widetilde V}

\renewenvironment{abstract}{%
	\quotation
	\small
	\textbf{\textit{\abstractname.}} 
}{\endquotation}

\begin{document}
	
\title{Theta function basis of the
  Cox ring of positive $2d$ Looijenga pairs}

\author{Sean Keel}
\address{Sean Keel, Department of Mathematics, 1 University Station C1200, Austin, TX 78712-0257, USA}
\email{keel@math.utexas.edu}

\author{Logan White}
\address{Logan White, Department of Mathematics, 1 University Station C1200, Austin, TX 78712-0257, USA}
\email{loganw23@utexas.edu}


\maketitle
\begin{abstract} \cite[24.7]{KY22} gives a conjectural construction of
  a canonical {\it theta function} basis for the regular functions on
  a partial minimal model of a smooth affine log CY variety with maximal boundary.
  Here we carry out the construction to give a canonical basis of theta functions
  for the Cox ring of a two dimensional Looijenga pair with affine interior,
  with structure constants naive counts of $k$-analytic disks in the total space
  of the universal deformation of the mirror (which, as this is dimension two,
  is isomorphic to the log CY surface itself), confirming
  the two dimensional case of conjecture \cite[24.11]{KY22}. 
  \end{abstract}



\setcounter{tocdepth}{1} 
\tableofcontents

\section{Introduction and Statement of results}
Let $(X,E)$ be a smooth Looijenga
pair with affine interior $U \coloneqq X \setminus E$. Let
$$
u:\cT \coloneqq \uSpec(\oplus_{L \in \Pic(X)} \cL) \to X
$$
be the universal torsor (where $\cL$ is the invertible $\cO_X$ module of
sections of $L$ ), and $G := \cT|_U \to U$ its restriction. $G$ is a smooth affine
log CY with maximal boundary, let

$V \coloneqq \Spec(A_{\Sk(G)})$, where

$$
A_{\Sk(G)} \coloneqq \oplus_{g \in \Sk(G)(\bbZ)} \bbZ \cdot \theta_g
$$
is the absolute mirror in the sense of 
\cite[1.6]{KY22}, where $\Sk$ indicates the Kontsevich-Soibelman essential skeleton,
see \cite[pg. 435]{KY20}. The canonical $T_{A_1(X,\bbZ)} \coloneqq \Spec \bbZ[\Pic(X)]$ action
on $G$ induces $w:V \to T_{\Pic(X)}$, and its tropicalisation
$w: \Sk(V)(\bbZ) \to \Pic(X)$. 

  Here we prove:

  \begin{theorem} \label{thm:simple}
  $\Cox(X)$ has a basis, parameterized by $\Sk(V)(\bbZ)$, homogeneous
  for the natural $\Pic(X)$ grading, with degree given by the tropicalisation
  $w:\Sk(V)(\bbZ) \to \Pic(X)$, canonical up to individual scaling.
  \end{theorem} 

  Each irreducible component $Z \subset E$ determines a point of $\Sk(G)(\bbZ)$
  (namely the divisor $u^{-1}(Z) \subset \cT$), and thus a basis element
  $\theta_{Z \in \Sk(G)} \in A_{\Sk(G)}$. Let $C \subset \Sk(V)$ be cut out by
  the inequalities $\theta^{\trop}_{Z \in \Sk(G)} \geq 0$, for all $Z \subset E$.

\begin{proposition} \label{cor:poly}
    Notation as above. For $L \in \Pic(X)$, let
    $P_L \coloneqq w^{_-1}(L) \cap C \subset \Sk(V)$.

    The canonical coordinate points of the linear system $\bbP(H^0(X,L))$ in \cref{thm:simple} are parameterized
    by the integer points $P_L(\bbZ)$. 
    In particular
    the dimension of the vector space of sections is the number of integer points
    in $P_L$.
  \end{proposition}

  \begin{remark} \label{rem:poly} $\Sk(V)$ is an integer piecewise affine space. Each
    choice of open algebraic torus in $V$ identifies this with the
    real vector space given by the co-characters. In this structure
    $C \subset \Sk(V)$ is a convex cone, and the slices $P_L$ are
    convex polygons, which, when $U$ is a torus is the polygon associated
    with the polarized toric surface  $(X,L)$. 
    \end{remark} 
  
  The Cox ring of a projective variety, together with its natural grading
  by the Picard group, is a full Mori theoretic invariant. It
  knows for example all the birational contractions, see \cite{HK96}.
  So more interesting
  than the vector space basis, are the structure constants for the
  multiplication rule in this basis (because it is the Cox {\it Ring} that
  really matters). These involve naive counts of punctured
  Berkovich disks in $V$, a precise statement requires some preliminaries,
  see \cref{rem:structure}. However for one deformation of $X$ there is a simple   statement: Let
  $X_e$ be the minimal resolution of the deformation of $X$ corresponding
  to $e \in T_{\Pic(X)}$ (see \cite[\S 5]{MLP}). We note that
  for any smooth deformation $(Y,E)$ of $(X,E)$, 
  $\Eff(Y) \supset \Eff(X)$ (there can be more effective classes on $Y$), with
  equality iff $Y \setminus E$ is affine. Define
  $$
  \Cox'(Y) \coloneqq \oplus_{L \in \Eff(X)} H^0(Y,L) \subset \Cox(Y)
  $$

  \begin{theorem} Notation as immediately above.
     $\Cox'(X_e)$ 
      has a canonical basis parameterized by $\Sk(V)(\bbZ)$ with
      structure constants non-negative integers, naive counts of punctured
      Berkovich disks  in $V$.
    \end{theorem}

    We have a similar statement for $X$, see \cref{rem:structure}.

    Next we give the background necessary to give full statements of
    our main results:

\begin{construction} \label{const:mirror} 
Let $(X,E)$ be a smooth two dimensional Looijenga pair (i.e.
$X$ is a smooth projective (necessarily rational) surface and $E \subset X$
is an anti-canonical cycle of rational curves) with $U \coloneqq X \setminus E$
affine. 

Let
$A_{\Sk(G)} \coloneqq  \oplus_{g \in \Sk(G)(\bbZ)} \bbZ \cdot \theta_{g}$ be the absolute mirror
algebra to $G$, \cite[1.6]{KY22}. $V \coloneqq \Spec(A_{\Sk(G)})$.

Each irreducible component $Z \subset E$
determines a point $Z \in \Sk(G)(\bbZ)$ (namely $u^{-1}(Z) \subset \cT$) and so
$\theta_{Z} \in A_{\Sk(G)}$. Let $C \subset \Sk(V)$ be defined by
$\theta_Z^{\trop} \geq 0$, over all $Z \subset E$. The $T_{A_1(X,\bbZ)}$ action
on $G$ determines $V \to T_{\Pic(X)}$ and its tropicalisation
$w: \Sk(V)(\bbZ) \to \Pic(X)$, see \cref{def:skVweight}. By the main result of \cite{LZ23},
$V \to T_{\Pic(X)}$ is modular, the universal family of \cite[4.24]{MLPv1}.
In particular each of  the quniversal (short for quasi-universal)
families $\bbX$ of \cite[3.1]{MLPv1},
see \cref{sec:quf}, determines a small resolution
$\tV \to V$ and partial minimal model $\tV \subset \bbX$, and so
the mirror construction of \cite{KY22} can be applied to $V$. Let
$\bbX \subset \obbX$ be an snc compactification, and 
$A_{\Sk(V),L_{\obbX}}$ be the associated mirror algebra \cite[18.3]{KY22}, where
(for any $X$) 
$L_X \coloneqq \bbZ[\Pic(X)^*]$. Here there is a natural isomorphism
$\Pic(\bbX) = \Pic(X)$. 
$V \to T_{\Pic(X)} \coloneqq \Pic(X) \otimes_{\bbZ} \bbG_m$
induces a natural $\Pic(X)$ grading on $A_{\Sk(V),L_{\obbX}}$ (see \cref{sec:ta}) with
each theta function
homogeneous with degree given by the tropicalisation $w: \Sk(V)(\bbZ) \to \Pic(X)$.
The free submodule 
$A_{C,L_{\obbX}} \subset A_{\Sk(V),L_{\obbX}}$, with basis elements corresponding
to integer points of $C$,  
is a graded subalgebra. For each 
$L \in \Pic(X)$ we have an
associated GIT quotient $A_{C,L_{\obbX}}//L \to T_{\Pic(\obbX)}$. See \cref{nota:GIT}
for an
explanation of our abbreviated notation.

Each $Z \subset E$
determines a natural {\it face} of $C$ (given by the additional
requirement $\theta_Z^{\trop} = 0$), the complement is the basis of an
ideal of $A_{C,L_X}$, which
together give a boundary to each GIT quotient, which we indicate by $\partial$.

Choose a splitting $T_{\Pic(X) = \Pic(\bbX)} \subset T_{\obbX}$.
\end{construction}

Let $\ADM \subset \Aut(\Pic(X))$ be the group of automorphism preserving
the intersection product, the nef cone, and the classes of the boundary
divisors. See \cite[4.2,4.5]{MLP}. We write $\bbG_m^E$ for the split algebraic
torus with one factor for each irreducible component of $E$. Each has a modular
equivariant action  on $\bbX \to T_{\Pic(X)}$ ($\ADM$ by changing the marking
of $\Pic$, $\bbG_m^E$ changing the marking of the boundary). 
\cref{thm:simple} follows immediately from a more precise modular version: 

\begin{theorem} \label{thm:modular} Notation as immediately above. 
  Let $L \in \Ample(X)$ be GIT general. The restriction
  $$
  (A_{C,L_X}//L,\partial) \to T_{\Pic(X)}
  $$
  is isomorphic to the quniversal family $\bbX \to T_{\Pic(X)}$ associated
  with $L$. $A_{C,L_X}$, with its natural $\Pic(X)$ grading is isomorphic to
  $\Cox(\bbX)$ (this is independent of the choice of $L$, see \cref{rem:sqm}).
  The theta function basis elements are independent of
  the above choice of splitting $T_{\Pic(X) = \Pic(\bbX)} \subset T_{\obbX}$ up to
  multiplication by individual characters, so give canonical
  coordinate sections in $\bbP(H^0(\bbX,M))$, for each 
  $M \in \Pic(X) = \Pic(\bbX)$.

  The modular action of $\ADM$ on $\bbX$ permutes these coordinate sections,
  the modular action of $\bbG_m^E$ fixes them.
  
Structure constants for the multiplication rule are naive non-negative integer
counts of punctured Berkovich disks in the total space of $V$. 
\end{theorem}

\begin{remark} In the theorem $H^0(\bbX,M))$ is a finite rank
  free $L_{\bbX}$ module,
  and $\bbP(H^0(\bbX,M))$ means the associated (trivial) projective space
  bundle over $T_{\Pic(X)}$.
  \end{remark}

\begin{remarks}
  As noted above, the Cox ring of projective variety (together with its canonical
  grading by the Picard group) knows all the birational contractions.
  The birational geometry of $\bbX$ is very rich, see
    \cite[\S 4]{MLPv1}. 
    By the above all its contractions are determined
    by naive counts of analytic disks on the universal deformation of $X$.
    The assumption
    of affine interior implies $\bbX$ is log Fano and so $\Cox(\bbX)$ is finitely
    generated, by deep results of \cite{BCHM}. Our mirror symmetric approach
    gives a quite different explanation for finite generation. Without
    the assumption of affineness of $U$, $\Cox(\bbX)$ need not be finitely
    generated and we do not know what to expect. However, we think the
    first paragraph of \cref{thm:modular} will hold, by a very similar
    argument. See \cref{thm:mpp} and \cref{rem:neg}. 

     When $L$ is ample, we expect the canonical theta function basis of $H^0(X,L)$
  could be obtained by some slight improvements on the methods of \cite{HPV24}
  (the log Kulikov models on which the paper is based depend on auxiliary choices,
  so some further argument is needed to show the theta functions obtained are canonical).
  But the
  main Mori theoretic force is in the ring structure of $\Cox(\bbX)$,
  and for this we
  need interaction between theta functions for different line bundles,
  we do not see
  how to approach this via \cite{HPV24}.

  The GIT fan for the action of $T_{A_1(X,\bbZ)}$ on $\Cox(\bbX)$ is,
  by \cite{HK96} identified with the Moving fan of $\bbX$. By \cite[4.14]{MLPv1}
  this is identified with the Weyl chambers given by the roots. 

\end{remarks}

\begin{remark}[Stucture constants for the mirror algebra] \label{rem:structure} For any snc compactification
$\obbX$ of a small resolution of an affine log CY, $V$,  with maximal boundary,
\cite{KY22} associates an algebra structure on the free $L_{\obbX}$ module
$$
A_{\Sk(V),L_{\obbX}} \coloneqq \oplus_{P\in \Sk(V)(\bbZ)} L_{\obbX} \cdot \theta_P
$$
with basis $\Sk(V)(\bbZ)$. The algebra structure
is determined by structure constants 
$$
\theta_{P_1} \cdot \dots \theta_{P_k} = \sum_{Q \in \Sk(V)(\bbZ)} \mu(P_1,\dots,P_n,Q) 
\cdot \theta_Q
$$
for $\mu(P_1,\dots,P_n,Q) \in L_{\obbX} \coloneqq \bbZ[N_1(\obbX,\bbZ)]$
defined by
$$
\mu(P_1,\dots,P_n,Q) = \sum_{C \in \NE(\obbX,\bbZ)} \eta(P_1,\dots,P_n,Q,C) Z^C
$$
where $\eta(P_1,\dots,P_n,Q,C) \in \bbZ_{\geq 0}$ is a naive (meaning
cardinality of a finite set) count of punctured Berkovich disks in $V$ with
class $C$, see \cite[1.2]{KY22} for the precise definition of the count. 

We apply this to the double mirror algebra $A_C$. Once we choose the splitting
$s: \Pic(\obbX)^* \twoheadrightarrow \Pic(X)^* = \Pic(\bbX)^*$, the structure constants for 
$A_{C,L_X}$ are the sum of the structure constants for $A_{C,L_{\obbX}}$ over
all classes in $s^{-1}(C)$ (only finitely many will be non-zero), and then
finally for a smooth deformation $(Y,E)$ of $(X,E)$, together with a marking of
$\Pic$ and the boundary, we have the associated marked period point, see
\cite[1.11.3]{MLP}, 
$\Phi: L_X \to \bbC$, and the structure constants for $\Cox'(Y)$ are given by
applying $\Phi$ to the constants for $A_{C,L_X}$. 
\end{remark} 

\paragraph{\bfseries Acknowledgments}
We benefited a lot from detailed conversations with Paul Hacking
and Pierrick Bousseau.
The paper is everywhere influenced by the philosophy and results of
our previous joint work with Tony Yu.
Both of us received support from NSF grant DMS-1561632.

Throughout the paper we work over an algebraically closed field $k$, with
trivial valuation. 

  \section{The quasi-universal families} \label{sec:quf} 

  In \cite[\S 3]{MLPv1} an elementary blowup (and blowdown)
  construction is given 
  of a family
  $(\bbX,\bbE) \to T_{\Pic(X)}$ with fibres the smooth Looijenga pairs
  deformation equivalent to $(X,E)$. The family is not unique, but any two
  are birational (over $T_{\Pic(X)}$) by a rational map that induces a fibre
  by fibre isomorphism -- this map is an SQM (small $\bbQ$-factorial
  modification) and as such corresponds to a maximal cone in
  $\MovFan(\bbX)$. \cite{MLPv1} identifies this fan with the hyperplane
  arrangement on $\Nef(X)$ given by the roots. See \cite[4.14]{MLPv1},
  \cite[3.2]{MLP}. 
  In the positive
  case (see \cite[4.20]{MLPv1}) there is a canonical blowdown $\bbX \to \obbX$ (over $T_{\Pic(X)}$)
  which on each fibre contracts all complete curves in the interior.
  $(\obbX,E) \to T_{\Pic(X)}$ is modular, an honest universal family of
  Looijenga pairs with affine interior with at worst du Val singularities
  away from the boundary, see \cite[4.24]{MLPv1}.  We refer to the $\bbX$ as quniversal (for
  quasi-universal) families.

  \begin{remark} \label{rem:sqm} Note as the various quniversal families are all small $\bbQ$-factorial
    modifications (SQMs) of one another, their Cox rings and Mori Fans are
    canonically identified.
    \end{remark}

Next we describe the idea behind the proof.

  \subsection{The Philosophy of the proof}   \label{rem:conjs}

  We think here dimension two, but the basic construction applies in general as
  long as the mirror $V$ (to $G$,  notation as in \cref{const:mirror}) has the right singularities for \cite{KY22} to apply, namely admits a small resolution. 
  \begin{remark} One could instead use \cite{GS}, which has no such
    singularity restrictions. But crucial for all our applications are basic
    convexity results, \cite[20.1,24.7]{KY22}, and \cite[15.6]{KY20}. These are elementary
    in the \cite{KY22} theory. But the analogs are not known (and to us
    appear difficult) in the \cite{GS} theory. \end{remark}
  
  We start with $(X,E)$ a smooth
  Looijenga pair. We assume affine interior $U := X \setminus E$, though
  much of the construction will work in general, see \cref{rem:small}.

  As above we take $u:\cT \to X$ the universal torsor and 
  and $G:= \cT|_U \to U$. $G \subset \cT$ is a {\it partial minimal
    model}, meaning the volume form of $G$ has a (necessarily simple)
  pole along each divisor in the complement $\cT \setminus G$.

  By
  definition
  $\cO(\cT) = \Cox(X) \coloneqq \oplus_{L \in \Pic(X)} H^0(X,L)$.
 
      \cite[\S 24, 24.1]{KY22} gives  a 
       conjectural scheme for producing a theta function basis of
      $\cO(\cT)$ for any partial minimal model $G \subset \cT$ of an affine
      log CY with maximal boundary (here we are using notation for the application,
      but this philosophy holds in general).

      A  theta function basis of $\cO(\cT) \subset \cO(G)$ would
      follow 
      from three basic conjectures:

\begin{conjecture}[Double Mirror] \label{conj:dm}
            $V^{\vee} \coloneqq (U^{\vee})^{\vee} = U$ up to coefficients, see \cref{rem:coeff}.
    \end{conjecture}
    \begin{conjecture}[Symmetry]
      $\theta_{u \in \Sk(u)}^{\trop}(v) =
      \theta_{v \in \Sk(V)}^{\trop}(u)$

      (See \cref{nota:subscript}  for
      an explanation of the notation. )
    \end{conjecture}

\begin{conjecture}[Independence]
    
  $(\sum a_v \Theta_v)^{\trop} =
  \min \Theta_v^{\trop}$,
  for any linear combination
  with non-zero coefficients. Note a priori we have $\geq$ rather than equal. 
\end{conjecture}

The recent preprint \cite{CMMM} proves a version of Symmetry and
Independence in the case of Fock-Goncharov mirror cluster varieties. Our $G$
and $V$ are such a pair by
\cite[4.4]{BGCA} and \cite[\S 3]{TravisCoxCluster}, and Independence and
  Symmetry, in the form we will need to prove \cref{thm:modular}, follow
  from \cite{CMMM}. So the main issue for us will be the double mirror
  conjecture, see \cref{cor:dm}.

\begin{remark}[Coefficients] \label{rem:coeff} Throughout there is the issue of coefficients:
    For a compactification $V \subset \oV$ (of an affine log CY $V$ with
    maximal boundary which
    admits a small resolution) the mirror algebra, $A_{V \subset \oV}$ is a free
    $\bbZ[\NE(\oV,\bbZ)] \eqqcolon R_{\oV}$ algebra,
    with basis $\Sk(V)$. Different
    compactifications give algebras with the same bases, related (essentially)
    by base extension. More precisely, when $V \subset \oV$ is snc, the
    $\Spec(A_{V \subset \oV}) \to \Spec(R_{\oV}) = \TV(\Nef(\oV))$ (note the
    target is an affine $T_{\Pic(\oV)}$-toric variety) is equivariant
    for the subtorus $T_{<D>} \subset T_{\Pic(\oV)}$, where
    $<D> \subset \Pic(\oV)$ indicates the subgroup generated by
    irreducible components of the boundary $\oV \setminus V$ (each
    theta function is an eigenfunction for $T_{<D>}$). 
    When we
    replace $\oV$ by $\oV' \to \oV$ a birational map which is an
    isomorphism over $V$, then, after restricting the structure tori
    of the base, the mirrors change by base extension:
    $$
    A_{V \subset \oV'} \otimes \bbZ[A_1(\oV',\bbZ)] =
    A_{V \subset \oV} \bigotimes_{R_{\oV}} \bbZ[A_1(\oV',\bbZ)].
    $$

    In particular, the set of fibres of the mirror family (restricted to the
    structure tori) are independent of the compactification (and the
    theta functions are independent, up to individual scaling, determined by
    the action of $T_{<D>}$). There are two versions of the mirror that
    eliminate the dependence on compactification, the {\it absolute mirror}
    \cite[1.6]{KY22}, and the {\it canonical mirror}, \cref{sec:canmir}.

    Here, to simplify notation, we will sometimes write e.g. $A_{\Sk(V)}$ and
    talk about the mirror {\it up to coefficients}. It is only when we
    are dealing with modularity questions that the precise choice of coefficients
    will matter, which we indicate with further decoration, e.g.
    $A_{\Sk(V),L_X}$ in \cref{const:mirror}.
    \end{remark}

Now given partial minimal model
$G \subset \cT$, with divisorial boundary $E \coloneqq \cT \setminus G$,
we consider $V := G^{\vee}$ (the mirror) and the {\it convex cone} $C \subset \Sk(V)$, given by $\theta_{Z \in \Sk(G)}^{\trop} \geq 0$ for all irreducible components
$Z \subset E$.

\begin{remark} We use italics in {\it convex cone} as in \cite{GHKK}. $C$ lives
  in $\Sk(V)$ which is a piecewise integer affine space, with no canonical
  affine structure, so the ordinary notion of convex does not make sense. 
   A basic idea from \cite{GHKK}, expanded in
  \cite{KY20}, \cite{KY22} is that intrinsic skeleta, e.g. $\Sk(V)$,  come with
  canonical paths -- the intrinsic spines of punctured Berkovich disks, which
  we can treat as straight in order to generalize many convex geometric notions; we
  use italics to indicate such generalisations. 
    If
  $G = T_N$, the algebraic torus with co-character lattice $N$, then 
  $\cT$ is a toric variety, $\Sk(V)(\bbZ) = N^*$, and
  $\bbZ[C] = \cO(\cT)$, and the spines are straight in the usual sense.
  $C$ in this case is the associated  rational polyhedral cone.
  
  \end{remark}

  \begin{notation} \label{nota:subscript}  As we will be using various 
    mirror algebras simultaneously, e.g. theta functions on both
    $G$ and its mirror $V$, we add further decoration to avoid confused,
    writing e.g. 
    $\theta_{Z \in \Sk(G)} \in A_{\Sk(G)(\bbZ)}$, where $A$ always means a
    subring of the mirror algebra and
    the subscript on $A$ indicates the subset of the skeleton we take as basis.
     \end{notation}

By the basic convexity, \cite[15.8]{KY20},  
the free submodule $A_C \subset A_{\Sk(V)}$  with this basis is a subalgebra. 
The three conjectures would imply this is (up to coefficients)
$\cO(\cT) \subset \cO(U)$. For the short proof of this
implication, see \cite[24.6]{KY22}.  

\begin{remark}
  The independence conjecture implies $C$ is equivalently defined 
  by the  single inequality using the superpotential, $(\sum \theta_Z)^{\trop} \geq 0$. 
\end{remark}

\begin{construction}[The boundary] \label{const:boundary}
The general philosophy also describes the ideal of the boundary
$\cT \setminus G \eqqcolon E$: Each irreducible $Z \subset E$ gives
a {\it face}  $Z^{\perp} \subset C$, the locus where (in addition to the inequalities
defining $C$), $\theta_{Z \in G}^{\trop} = 0$.  Warning: the faces ARE NOT in general
{\it convex}, e.g. the submodule $A_{Z^{\perp} \subset C} \subset A_C$ will not
in general be a subalgebra. 
But IT IS the basis of an algebra: By the basic convexity, the
complement $[Z > 0] \subset C$ is the basis of an ideal of $A_C$, and so
$Z^{\perp} \cap C$ is the basis of the algebra $A_C/[Z > 0]$. The basic
conjectures imply that $[Z > 0] \subset A_C = \cO(\cT)$ is the ideal of
functions vanishing on $Z$.

\begin{remark} \label{rem:convexcones} It is not important above that
we used theta functions. For any regular $f$ on affine log CY $V$ (with
maximal boundary), the {\it cone}  $f^{\trop} \geq 0 \subset \Sk(V)$ is the basis of a subalgebra,
and $[f^{\trop} > 0] \subset [f^{\trop} \geq 0]$ is the basis of an ideal,
and the face $f^{\trop} = 0$ is the basis of the algebra  
$A_{[f^{\trop} \geq 0]}/[f^{\trop} > 0]$. There is a similar construction
for a set of functions (the construction we use above for the
$\theta_{Z \in \Sk(G)}$, $Z \subset E$). 

The 
    structure constants for $A_C/[Z > 0]$ also count disks: this
    quotient is the {\it mirror algebra} for the generic fibre of the
    function, see \cite[24.7]{KY22}. Italics because this fibre does not have
    to be log CY. However it inherits a volume form from the log CY, so
    a skeleton, and counts of disks give structure constants for a (a priori
    not necessarily associative) algebra, and by \cite[24.7]{KY22} this
    algebra is exactly $A/I$ (so in particular associative). 
    We expect that all strata in partial minimal models, and all log
    canonical centers come by a version of
    this construction. This 
    would then imply the Independence Conjecture in \cref{rem:conjs}.
    The description of the quotient as mirror algebra for the 
general fibre will give one of the two basic inductions  we will use to prove
\cref{thm:modular}, see \cref{rem:inductions}.
\end{remark}
\end{construction}

\subsection{dual cones} \label{sec:dcs}
  In toric geometry we have convex cones and their duals. We
  generalize this as follows: Suppose we have $G \subset W$ a partial minimal
  model with $W$ proper over affine. We have an associated analytic domain
  $W^{\beth} \subset G^{\an}$, these are the valuations with center in $W$, or
  equivalently given by $g^{\trop}(v) \geq 0$ for all $g \in \cO(W)$. We note this
  is the Berkovich generic fibre, for $W$ over $k$ with trivial valuation.
  $W^{\beth} \cap \Sk(G)$ is then {\it a convex cone},
  meaning the vector subspace $A_{W^{\beth} \cap \Sk(G)} \subset A_G$ is a subalgebra.
  And $W^{\beth}$ itself is convex in that structure constants for $A_{W^{\beth} \cap \Sk(G)}$
  can be computed with disks in $W^{\beth}$, see \cite[\S 3]{HKY22}.
  Now we have a dual construction, we can consider the mirror
  $V = \Spec(A_G)$. Each irreducible component $Z \subset W \setminus G$ gives
  $Z \in \Sk(G)$ and so $\theta_{Z \in \Sk(G)} \in \cO(V) = A_{\Sk(G)}$. Then we
  can take $C \subset \Sk(V)$ given by $\theta_Z^{\trop} \geq 0$. This is again
  convex, giving subalgebra $A_C \subset A_{\Sk(V)}$. This gives $V^{\vee} \to
  \Spec(A_C)$, which conjecturally is (up to coefficients) $G \subset W$.

\subsection{torus action in the general philosophy} \label{sec:ta}
When a torus $T_A$ acts on $G$ (in our application $G$ is a principal
$T_{A_1(X,\bbZ)}$ bundle) the mirror inherits, by \cite{LoganThesis},
a map
$V \to T_P$ to the dual torus (in our application $T_{\Pic(X)}$). This in
turn, by the basic convexity \cite[15.8]{KY20},
induces a $T_A$ action on the double mirror $V^{\vee}$, or
equivalently, a $P$-grading on $A_{\Sk(V)}$:

The tropicalisation of $V \to T_p$ is $w: \Sk(V) \to P$, 
\begin{definition} \label{def:skVweight}
      For $v \in \Sk(V)$, the weight $w(v): A \to \bbZ$ is
      $C \to (z^C)^{\trop}(v)$, for $W \in A$, where $z^W: T_P \to \bbG_m$
      is the corresponding character.
    \end{definition}
By the basic convexity , \cite[15.8]{KY20}, these are the weights for a $P$-grading 
on $A_{\Sk(V)}$ (so that the $L$-homogenous part has basis $w^{-1}(L)$), with
$A_C \subset A_{\Sk(V)}$ a graded subalgebra. Now we can consider the GIT
quotients for $T_A$-acting on $\Spec(A_{\Sk(V)})$, for the linearization of
the trivial bundle given by each character $L \in P$. By the basic
convexity, \cite[24.7]{KY22}, the coordinate ring (i.e. the direct sum of weight spaces for multiples
of $L$) is the mirror to the restriction of $V \to T_{\Pic}$ to a general
orbit of the one parameter subgroup of $T_P$ corresponding to $L$. 

 \begin{notation} \label{nota:GIT} As we will not vary the torus by which we
        quotient, but frequently vary the linearization,  we adopt a simplified notation:
        For $L \in \Pic(X)$, we will write
        $A_C//L$ for the GIT quotient $\Spec(A_C)^{\mss}(L)//T_{A_1(X,\bbZ)}$. 
        \end{notation} 

  We now employ this philosophy to prove \cref{thm:modular}.

        \section{Strategy of the proof}
         Our goal is, roughly, to prove that the fibres of the
        double mirror         
        $A_C \subset A_{\Sk(V)}$ give the universal deformation
        of $\Cox(X) \subset \cO(G)$. We see $X$ (and its deformations) then
        as GIT quotients. General properties of the mirror construction
        will guarantee these GIT quotients are families of Looijenga pairs,
        with at worst du Val singularities off the boundary (and, we will show,
        smooth near the boundary).  See \cref{prop:lpsings}.
        A basic
        question is how we know we are getting {\it the right} Looijenga pairs,
        e.g. pairs deformation equivalent to $(X,E)$. 
        For the interiors we get a strong statement nearly for free from
        \cite{LZ23}, see \cref{sec:pp}, \cref{cor:versal}.
        So the interior are {\it right}. How do we know we
        are getting the right compactification? For this we use the philosophy
        of \cite{HK96}: variation of GIT quotient (VGIT) applied to the
        $T_{A_1(X,\bbZ)}$ action on the Cox ring recovers all of the birational
        contractions of $X$. In particular it knows about any toric model.
        So we take a convenient toric model for $X$. This canonically
        induces, through the mirror and VGIT
        formalism, corresponding birational maps of GIT quotients, which
        we check have the correct deformation type. See \cref{cor:tm},
        \cref{prop:induct}. 
        Connecting this with what
        we already know about the interiors shows we have the right family
        of Looijenga pairs, see \cref{prop:LZ} and \cref{thm:mpp}.

        The key point in the construction is the following: 
        Note if we have a contraction $X \to \oX$ (we will
        mainly be interested in the blowdown of some configuration of disjoint
        internal $-1$ curves), then naturally $\Cox(\oX)$ is a subring of
        $\Cox(X)$. So we expect that the subring of $A_C$ for $T_{A_1(X,\bbZ)}$-weights
        in $p^*(\Pic(\oX)) \subset \Pic(X)$ is the corresponding ring
        $A_C^{\oX}$ (obtained by the same construction with $X$ replaced by $\oX$).
        So we expect we can obtain this subring by induction on the charge (see \cref{ss:ind}).    This is indeed the case:
        By the basic formalism this subring is determined by counts
        of disks in an analytic domain determined by the weights, roughly
        the 
        closure of a general $T_{p^*{\Pic(\oX)}}$ orbit. For the precise
          statement see \cref{prop:bbY}. The
        basic fact, proven in \cite{GHKIv2}, is that this formal deformation
        is isomorphic to the GHKI family for $\oX$ (restricted to the
        formal neighborhood of $0 \in \TV(\Nef(\oX))$ ). See  \cref{prop:eqdef}.

 \section{Induction on the charge} \label{ss:ind}
        Now we turn to the proof of \cref{thm:modular}. 
  
        We induct on the {\it charge}
        of $U$, the relative Picard number of a toric model for any smooth
        Looijenga pair with $U = X \setminus E$. 

          For charge at most one $\Pic(U)$ is
        trivial. In this case \cref{thm:modular} is 
        quite simple, and we treat it by a special argument
        in \cref{sec:prod}. See \cref{cor:crtp}

        Formal properties of the construction show that we are free to
        replace $X$ by a toric blowup (composition of blowups at nodes of
        the boundary), see \cref{prop:tbred}. The following is easy:

        \begin{lemma} \label{lem:ch-1} Let $(X,E)$ be a Looijenga pair with charge    at
          least two. Then there is a toric blowup $X' \to X$, which has
          two disjoint internal $-1$ curves, meeting different irreducible
          components of $E' \subset X'$. These can be simultaneous
          blowndown, to a smooth Looijenga pair.
        \end{lemma}

        Here {\it internal} means a smooth rational curve (necessarily
        of self intersection $-1$) meeting the boundary transversely in
        a single point. Equivalently, a $-1$ curve which is not an
        irreducible component of the boundary.

        Similarly:

        \begin{lemma} \label{lem:gen} Let $(X,E)$ be a smooth Looijenga pair. $(X,E)$ has a toric
            blowup whose Picard group is generated by the pullbacks of the Picard
            groups of each toric model.
          \end{lemma}
          \begin{proof} It is easy to see that $X$ has a toric blowup whose
            Picard group is generated by pullbacks of the Picard groups of
            all blowdowns of internal $-1$ curves , e.g. if $X$ has charge
            at least two, there is a toric blowup with two disjoint internal
            $-1$ curves, and the Picard groups of these two blowdowns suffice.
            A similar argument works for charge one. 
            Now proceed by induction on the charge.
            \end{proof} 

\subsection{toric blowup} 
        The next proposition shows that to prove the main theorem
        we are free to pass to toric blowups:

        Let $b:X' \to X$ be the blowup of a node. We have the analogous
        family version $\bbX' \to \bbX$ (blowing up this node on each fibre).
        Note we have canonical $A_1(X',\bbZ) = A_1(\bbX',\bbZ)$.

\begin{proposition} \label{prop:tbred} Notation as immediately above. 
          Let
          $T \subset T_{A_1(X',\bbZ)}$ be the subtorus corresponding
          to the $-1$ curve $[\ex(b)] \in A_1(X',\bbZ) = A_1(\bbX',\bbZ)$. 
          The following hold:
          \begin{enumerate}
          \item $\Cox(\bbX')^T = \Cox(\bbX) \otimes_{L_\bbX} L_{\bbX'}$.
          \item $(A_{C,L_{\bbX'}})^{T}= A_{C,L_\bbX} \otimes_{L_\bbX} L_{\bbX'}$
          \end{enumerate}
          
          If \cref{thm:modular} holds for $X'$, it
          holds for $X$. 
        \end{proposition}

        \begin{proof}  (1) is immediate
          from the definitions, and the projection formula. 

          For (2): By the basic convexity \cite[20.1]{KY22}, $(A_C^{X'})^T$
          (with whatever coefficients) is computed using disks in the restriction
          of $V \to T_{\Pic(X')}$ to the kernel of $z^{\ex(b)}$, ie the subtorus
          $T_{\Pic(X)} \subset T_{\Pic(X')}$. By \cite[\S 17]{KY20}, this is the GHKI
          mirror family for $(X,E)$ (or alternatively, we can deduce this
          restriction statement from the \cite{LZ23}
          identification of the GHKI mirror with the universal family of
          \cite[4.24]{MLPv1} -- for the universal family the restriction statement
          is easy). There is a natural $\bbX'|_{T_{\Pic(X)}} \to \bbX$, which on
          each fibre is the blowup of the analogous node. Now (2) follows from
          \cref{prop:cq}. 

          The implication in the final
          paragraph follow easily from(1-2).  
          
        \end{proof}

\section{log canonical centers of the GHKI mirror family}
         
\begin{proposition} \label{prop:fibsing} We consider the restriction $\ocY \to T_{\Pic(X)}$ of
  the GHKI family. The following hold:
 
  The total space of the restriction to the orbit of any subtorus has
  $K$-trivial Gorenstein canonical singularities. For fixed subtorus, and
  generic orbit, the total space admits a small resolution. $\ocY$ itself
  admits a small resolution.
\end{proposition}
\begin{proof} By \cite{LZ23} this is the universal family from \cite{MLP}.
  Then each quniversal family, see \cref{sec:quf}, gives a small resolution, and everything
  follows easily from the explicit blowup (and blowdown) construction of quniversal
  families in \cite{MLP}.
  \end{proof}

  \begin{proposition} \label{prop:dvres} Consider $p: X \to \oX$ the
    blowdown of a set of (necessarily disjoint)
    internal $-1$ curves (we allow here the empty set, ie $p$ can be the identity) and 
    $L = p^*(A)$ with $A$ ample.
    Let $\ocY \to \bbA^1_L$ be the restriction
    to an orbit closure for the associated one parameter subgroup. The
    following hold:
    \begin{enumerate}
    \item     The central fibre is the vertex $\VVert(X,E)$,
      see \cite[22.3]{KY22}.
      \item The pair $(\ocY,\VVert)$ is log canonical.
        \item Every log canonical center (see \cite[4.15]{Kollar_Singularities_of_the_minimal_model_program}) of
          the pair is a stratum of $\VVert$.
          \end{enumerate} 
        \end{proposition}

        In the statement, $\VVert(X,E)$ is the spectrum of the
        Stanley-Reisner ring associated to the dual complex of $(X,E)$.
                
  \begin{proof} That $\VVert$ is the central fibre follows easily from
    the structure constants, see e.g. \cite[22.4]{KY22}. The pair is
    canonical off the central fibre by \cref{prop:fibsing}, so all
    lc centers have image in $\VVert$. Of course the center has to be
    part of $\Sing(\VVert)$ (outside here the total space is smooth since
    $\VVert \subset \ocY$ is Cartier). So we just have to check there are
    no zero dimensional centers, along $1$-strata of $\VVert$. But now
    this follows from the explicit local charts in \cite{GHKIv2}:
    we have
    local charts around the (interiors of) $1$-strata for any such $L$, see
    \cite[0.10,3.15]{GHKIv1} and  \cite[2.5,2.7,]{GHKIv2}. 
  \end{proof}

  The cone $C \subset \Sk(V)$ has a {\it boundary}, $\partial C$,
  defined as the union of {\it faces}, where there is one face
  for each irreducible component $Z \subset E$, defined as
  $$
  Z^{\perp} \coloneqq (\theta_{Z \in \Sk(G)}^{\trop} = 0) \cap C. 
  $$
  
  We will need to know this is {\it small} in the following sense:

  \begin{proposition} \label{prop:small} Notation as immediately above. No face $Z^{\perp} \subset C$
    contains an open (in $\Sk(V)$) neighborhood of a point in $w^{-1}(\Nef(X))$.
  \end{proposition}
  \begin{proof} $w:\Sk(V) \to \Pic(X)_{\bbR}$ is open. So it's enough to prove
    $Z^{\perp}$ contains no open subset of the fibre $w^{-1}(A)$ for
    $A \in \Ample(X)$. This we compute as follows: Take $F_A \subset \Pic(X)$
    the fan with just the single ray $\bbZ \cdot A$. Then there is
    a natural $f: \TV(F_A) \to T_{Pic(X)/\bbZ A} \eqqcolon T$,
    a trivial  $\bbA^1$ bundle. 
    We pullback the GHKI family along $\TV(F_A) \to \TV(\Nef(X))$ to give
    $s: \cY \to \TV(F_A)$
    Rational points of  $w^{-1}(A) \cap C \subset \Sk(V)$ correspond to
    centers of log canonical singularity for the pair $(\cY,\cY|_B)$, where
    $B \subset \TV(F_A)$ is the boundary divisor, which dominate $B$.
    The restriction $\cY|_B \to B$ is trivial, with fibre $\VVert$ as in
    \cref{prop:dvres}.
    Points of $Z^{\perp} \cap w^{-1}(A) \cap C$ correspond to exceptional
    divisors (on some resolution of the total space) which do not map
    to $0 \times B \in \VVert \times B = \cY|_B$. It is easy to see this
    locus is one dimensional: By \cref{prop:dvres}, the exceptional
    divisor is either the strict
    transform of a component of $\VVert$ (which gives a single point of $\Sk(V)$),
    or dominates $S_Z \times B$, where $S_Z \subset \VVert$ is the one stratum
    corresponding to $Z \subset E$. The deformation generically along
    $S_Z \times B$ is toric, given by the toric charts of \cite{GHKIv2}, and one
    computes that these valuations contribute an interval to $Z^{\perp} \cap w^{-1}(A) \cap C$ (the dual complex to the exceptional locus in the
    minimal resolution of the surface singularity $xy = t^n \subset \bbA^3_{x,y,t}$).
    The fibre $w^{-1}(A) \subset \Sk(V)$ is two dimensional.
    This completes the proof.
    \end{proof}

        \section{local models}
        We have the mirror family $V \to T_{Pic(X)}$, $\bbV \subset \bbX$
        gives a fibrewise compactification of a small resolution. Let
        $\bbX \subset \bbX_1$ be an snc partial compactification with
        $\bbX_1 \to \cY$ regular and proper, where $\cY \to \TV(\Nef(X))$ is
        the GHKI mirror family. Finally let $\bbX_1 \subset \obbV$ be
        an snc compactification.

        We take $p: X \to \oX$ contracting a set of (necessarily
        disjoint) internal $-1$ curves. We consider the subring
        $A_{C,p^*(\Nef(\oX)),L_{\obbV}} \subset A_{C,L_{\obbX}}$ generated by
        homogeneous elements with weight in $p^*(\Nef(\oX))$.

        Let $\oT \subset \TV(\Nef(X))$ be the closure of a general
        $T  \coloneqq T_{p^*(\Pic(X))}$ orbit. We note $T \subset \oT$ is
        the toric variety 
        $T_{\Pic(\oX)} \subset \TV(\Nef(\oX))$, and the intersection of
        $\oT$ with $S_p$, the closed stratum associated to $p$, is
        (under the natural isomorphism $\oT = \TV(\Nef(\oX)$) 
        the zero stratum $0 \in \TV(\Nef(\oX))$, and this is
        an interior point of  $S_p$. 

        \begin{lemma} \label{lem:vert} The restriction $\cY|_{S_p} \to S_p$ is trivial,
          with fibre the vertex $\VVert$ for $\Sigma_{(X,E)}$.
        \end{lemma}
        \begin{proof} See \cite[3.40]{GHKIv2}. One can also deduce this directly
          from the definition of structure constants, see \cite[22.4]{KY22}.
          \end{proof}

        \begin{proposition} \label{prop:bbY} Let $\bbY$ indicate the formal completion
          of $\cY|_{\oT}$ at the point $0 \in \VVert$, over the zero
          stratum of $\oT = \TV(\Nef(\oX))$. The subring
          $A_{C,p^*(\Nef(\oX)}$ can be computed using analytic disks in
          the Berkovich generic fibre of $\hcY$ (viewed as special
          formal scheme over the constants).

          More precisely: Choose snc compactification
          $\bbV \subset \obbV$ as above. Let $\hbbY \to \bbY$ be a log
          resolution of this affine formal scheme. Assume the composition
          $\hbbY \to \cY$ factors through $\obbX_1 \to \cY$
          (see the definition of $\bbX_1$ above).
          Then
          $$
          A_{C,p^*(\Nef(\oX)),L_{\obbV}} = A_{\hbbY \to \bbY}
          \otimes_{L_{\hbbY/\bbY}} L_{\obbV}
          $$
          where $A_{\hbbY \to \bbY}$ is the mirror to the formal
          deformation, see \cref{sec:mfd}. 
        \end{proposition}
        \begin{proof} By the basic convexity \cite[24.7]{KY22}, we can
          compute with disks in the closure of general $T_{\Pic(\oX)}$ orbit.
          Further, 
          by definition of $w: \Sk(V) \to \Pic(X)_{\bbR}$,
          for $v \in  w^{-1}(\Nef(\oX)) \subset \Sk(V)$, $v(z^C) \geq 0$
          for all $C$ in the dual cone to the face $p^*(\Nef(\oX)) \subset
          \Nef(\oX)$, or equivalently, the center of $v$ maps into the
          affine open subset $\TV(p^*(\Nef(\oX))) \subset \TV(\Nef(\oX))$.
          As $v \in \Sk(V)$, the center is in the boundary, so the image
          is in the closed stratum $S_p$. It follows we can compute
          structure constants with disks in the formal completion of
          $\cY|_{\oT}$ along the fibre over $0 \in \oT = \TV(\Nef(\oX))$.
          We argue that we can work more locally, namely at the further
          completion (of the total space $\cY|_{S_p}$) at
          $0 \in \VVert(X,E)$ (see \cref{lem:vert}). This amounts
          to showing we can use structure disks in the locus where
          $\theta_{Z \in \Sk(G)}^{\trop} > 0$ for all $Z \subset E$. But
          we can compute with structure disks having transverse spine --
          which means in particular that the basepoint avoids any fixed
          lower dimensional locus in $\Sk(V)$, and so by \cref{prop:small},
          we can assume this in the desired locus (the complement of the
          faces of $C$). But now the full disk lies in this locus, by
          the maximum principal. This completes the proof.
        \end{proof}

      \begin{remark} \label{rem:small} The proposition shows that for these {\it pieces} of
        the mirror algebra, we need only the formal smoothing
        of the vertex, so we can hope to carry out the same argument when
        $U$ is not affine (and so the GHKI mirror is only
        formal). 
        \end{remark}

          \begin{proposition} \label{prop:eqdef}
            Let $\bbY$ be as in \cref{prop:bbY},
            and $\hTV(\Nef(\oX))$ be
          the formal completion at $0 \in \oT = \TV(\Nef(\oX))$. Then
          the formal deformation $\bbY \to \hTV(\Nef(\oX))$ is isomorphic
          to the formal completion of the GHKI family for $(\oX,E)$ at
          $0 \in \VVert(X,E)$ in its central fibre.
        \end{proposition}
        \begin{proof} This follows from the argument for \cite[3.40]{GHKIv2} --
          that proposition gives this in the case that $p: X \to \oX$
          is a toric model, but
          the argument works, without change, for any contraction of internal
          $-1$ curves. \end{proof} 

 \begin{remark} \label{rem:mw}  The equality of deformations in \cref{prop:eqdef} 
    is an instance of the {\it moving worms} of \cite{Kontsevich_Affine_structures}. The family over
    this stratum is trivial, but the singularities of the total space vary:
    over $0 \in S_{p}$ there is a complicated singularity at $0 \in \VVert$.
    But this factors over $S_{p}$, into (generically) simple singularities,
    along the $1$-strata of $\VVert$ with position given by the coordinates
    $z^Q$, for $Q$ the exceptional curves. In particular, over the (relative) interior
    of the stratum, 
    for each $Q$ has deformed away from $0 \in \VVert$, and the resulting singularity
    at $0$ is the singularity from the $\oX$ family.
    \end{remark}

        \begin{remark} \label{rem:beinduct} We will use \cref{prop:bbY} and
          \cref{prop:eqdef} to compute GIT quotients $A_C//L$ for
          $L \in p^*(\Nef(\oX))$, or maps of GIT quotients, such as
          $q_1,q_2$ in \cref{prop:induct}, by induction on the charge -- reducing the question
          from $X$ to $\oX$. But the statement we get is slightly
          subtle: There are three
          mirror families here, which we will refer to as the $X$ case, the
          $\oX$ case, and the formal case.

          We would like to say $A_{C,p^*(\Nef(\oX))}$ (the $X$ case) is 
          a base extension of $A^{\oX}_{C,\Ample(\oX)}$ (the $\oX$ case,
          where the superscript
          $\oX$ means the analogous double mirror algebra for $(\oX,E)$), but
          what we actually get, by \cref{prop:bbY} and \cref{prop:eqdef}
          is that each of these is a base extension
          of the same mirror algebra associated to (a sufficiently blownup)
          formal resolution $\hbbY \to \bbY$ as in \cref{prop:bbY} (the formal
          case). We do
          not know that $T_{\Pic(\oX)} \to T_{\hbbY/\bbY}$ is dominant, a priori
          the mirror family for the formal resolution could be {\it bigger}, could
          have new isomorphism classes of fibres. However this does show
          that the fibres of these three mirror families are deformation equivalent,
          and for our purposes, this will be enough.  We note there is one instance
          where the isomorphism classes of
          fibres of these families are the same, that is when $p$ is a full
          toric model. For in that case, the fibres are toric, and so have
          no moduli.
          \end{remark}

    \subsection{The mirror to a formal deformation} \label{sec:mfd}

In \cite{LoganThesis} the mirror construction of \cite{KY22} is extended to
the formal setting: Let $(\cY,\cD)$ be an affine formal scheme with an
effective Weil divisor such that $K_{\cY} + \cD$ is log canonical and
linear equivalent to an effective divisor $\cE$ supported
on $\cD$. Let $r: \hcY \to \cY$ be a log resolution of $(\cY,\cD)$ such that
the essential dual complex has maximal (namely $\dim \cY$) dimension. The main
result of \cite{LoganThesis} is an algebra structure on the free
$L_{\hcY/\cY}$ algebra $A_{\hcY \to \cY}$ with basis the integer points of the
essential dual complex, with structure constants naive counts of Berkovich
analytic disks on the Berkovich generic fibre of $\hcY$ (viewed as a special
formal scheme over the trivially valued ground field) as in \cite{KY22}.

 \section{GIT Basics} \label{sec:GITbasics}

 We recall some of the VGIT theory of a torus acting on an affine variety over $\mathbb{C}$, varying linearizations of the trivial bundle. Everything in
 this section is well known, we include a few short arguments for the
 reader's convenience.

Take $P$ an integer lattice of finite rank and $T^P$ the torus with character lattice $P$. Let $A$ denote a finitely generated integral $\bbC$-algebra, graded over $P$.
A linearization of $\Spec(A) \times \mathbb{A}^1$ corresponds to a choice of weight $\chi \in P$ describing the action of $T^P$ on the second factor. This linearization, call it $L_{\chi}$, has invariant sections $A_{\chi}$.

In \cite{BH06}, Berchtold and Hausen give a self-contained treatment of the VGIT theory of such linearizations. They show there is a finite rational polyhedral fan $\Sigma$ in $P$, with support a convex cone $|\Sigma| \subset P_{\bbR}$, where
$|\Sigma|(\bbZ) \subset P$ is 
the set of weights $\chi$ with $A_{\chi} \neq \emptyset$, and for any two $\chi_1, \chi_2 \in P$, we have equality of semistable loci $A^{\text{ss}}(L_{\chi_1}) = A^{\text{ss}}(L_{\chi_2})$ if and only if $\chi_1, \chi_2$ lie in the interior of the same cone $\sigma \in \Sigma$. Thus to each such cone is associated a semistable locus $A^{\text{ss}}(\sigma) \subset \Spec(A)$. Additionally, for any two $\sigma_1, \sigma_2 \in \Sigma$, we have an order-reversing correspondence between cone inclusions and semistable loci $\sigma_1 \subset \sigma_2 \iff A^{\text{ss}}(\sigma_2) \subset A^{\text{ss}}(\sigma_1)$.

Let $E = |\Sigma|(\bbZ)$.

\begin{lemma} \label{lem:finite} $E^{\text{gp}} = P$ iff the kernel of the action of
  $T^P$ on $\Spec(A)$ is finite.
\end{lemma}
\begin{proof} The kernel is the intersection of the kernel of all the
  characters in $E$. The lemma follows. 
  \end{proof}

  \begin{proposition} \label{prop:interiors} Assume $E^{\gp} = P$. 
   Then $A^s(\sigma) = A^{\text{ss}}(\sigma)$ for
   $\sigma \in \Sigma$ maximal.
   \end{proposition}
    \begin{proof}
      Choose non-zero homogeneous generators for $A$. This embeds 
      $\Spec(A)$ equivariantly as a closed subvariety of affine space $\bbA^n$,
      intersecting $\bbG_m^n \eqqcolon T^N$, so $T^P$ acts diagonally, via
      a homomorphism $N \rightarrow P$, which has finite cokernel by
      \cref{lem:finite}. 

      Now the VGIT picture for the action of a subtorus of $\bbG_m^n$ on $\bbA^n$ is studied in great detail in Cox, Little, and Schenck [2]. In
      particular, \cref{prop:interiors} holds there, see \cite[14.3.14]{CLS}.
      Let $\Sigma'$ indicate the corresponding GIT fan. 
      
      Note for any linearization
      $$
      (\bbA^n)^{\text{ss}}(\chi) \cap \Spec(A) = \Spec(A)^{\mss}(\chi)
      $$
      and the analogous statement holds for stable loci.
      We note the semi-stable locus for this action (for any linearization)
      is open and invariant under $\bbG_m^n$, so if its non-empty it
      contains $\bbG_m^n$, and thus intersects $\Spec(A) \subset \bbA^n$.
      It follows that $|\Sigma'| = |\Sigma|$, and that $\Sigma$ is a coarsening
      of $\Sigma'$ (where a cone in $\Sigma$ is the union of those cones
      in $\Sigma'$ whose semi-stable loci have the same intersection with
      $\Spec(A)$. Now the proposition for $\Sigma$ follows from the
      proposition for $\Sigma'$, since the locus of linearizations for which
      a given point is stable is convex, see \cref{cor:Mum} below.
      This completes the proof.
\end{proof}

\begin{proposition} \label{prop:bound} Say $\sigma_1, \sigma_2 \in \Sigma$ are maximal cones, and $\chi_i \in \sigma_i^{\circ}$, $i = 1,2$. The intersection of semistable loci gives birational $A^{\mss}(\sigma_1) \dashrightarrow A^{\text{ss}}(\sigma_2)$, further inducing a canonical birational map $A//L_{\chi_1} \dashrightarrow A//L_{\chi_2}$. The exceptional locus of this birational map is contained in the vanishing of the homogeneous ideal of the coordinate ring $\oplus_{k \geq 0} A_{k \chi_1}$ generated by elements of $A_{k \chi_1} \cap A \cdot A_{m \chi_2}$ varying over $k,m \geq 1$.
\end{proposition}  
\begin{proof}
    Fix some $k,m \geq 1$, and a homogeneous $s \in A_{k \chi_1} \cap A \cdot A_{m \chi_2}$. We have $V(A \cdot A_{m \chi_2}) \subset V(s)$, so the complement of $V(s)$ in $A^{\text{ss}}(\sigma_1)$ is contained in the intersection $A^{\text{ss}}(\sigma_1) \cap A^{\text{ss}}(\sigma_2)$. Thus the complement of $V(s)$ in $A // L_{\chi_1}$ is contained in the regular locus of the birational $A//L_{\chi_1} \dashrightarrow A//L_{\chi_2}$. 
\end{proof}

For the next proposition, we use the following consequence of [3, Proposition 2.5]. 
\begin{proposition}[Mumford's Criterion for Stability] \label{prop:Mum} Suppose $T^P$ acts on $\mathbb{C}^n$ as a subtorus of $(\mathbb{C}^*)^n$. Then a closed point $x$ is stable for $L_{\chi}$, $\chi \in P$, if and only if for every 1-parameter subgroup $\lambda$ of $T^P$ such that $\lim_{t \rightarrow 0} \lambda(t) \cdot x$ exists, we have $(\chi, \lambda) > 0$. 
\end{proposition}

We have the immediate:
\begin{corollary} \label{cor:Mum}. Notation as in 
\cref{prop:Mum}. The set of characters for which $x$ is stable is the set of integer points of a convex cone in $P_{\mathbb{R}}$. 
\end{corollary}

\begin{proposition} \label{prop:adj} 
  Say $\sigma_1,\sigma_2 \in \Sigma$ are adjacent maximal chambers sharing a codimension 1 wall $\tau$. Pick $\chi_i \in \sigma_i^{\circ}, i = 1,2$ and $\chi \in \tau^{\circ}$. Then $A^{s}(\sigma_1) \cap A^s(\sigma_2) \subset A^s(\tau)$, and the natural map $A// L_{\chi_i} \rightarrow A//L_{\chi}$ is an isomorphism away from the upper bound from \cref{prop:bound}.
  \end{proposition} 
\begin{proof}
     As in the proof of \cref{prop:interiors}, we embed $\Spec(A)$ equivariantly in some $\bbA^n$ for which $T^P \subset (\bbG_m)^n$ is a subtorus. Now by Mumford's criterion, the set of characters for which a fixed point is stable is convex, and we obtain the first assertion.

     For the second assertion: notice $A^s(\sigma_1) \cap A^s(\sigma_2)$ is open, contained in the stable loci for $\sigma_1, \sigma_2,$ and $\tau$. For any character $\chi$, the quotient $A^{\text{ss}}(L_{\chi}) \rightarrow A//L_{\chi}$ is an open map on its restriction to the stable locus of $\chi$. By uniqueness of categorical quotients, the images of $A^s(\sigma_1) \cap A^s(\sigma_2)$ in all $A//L_{\chi_i}$, $A//L_{\chi}$ are isomorphic open subschemes, mapped isomorphically onto one another by the various birational maps. As we observed in Proposition 2, the complement of the upper bound is contained in $A^s(\sigma_1) \cap A^s(\sigma_2)$. 
\end{proof}

\section{The period point} \label{sec:pp}
\cref{thm:modular} contains various modularity statements for
the double mirror $A_C$. In this section we deduce one, essentially for
free, from general properties of the mirror construction, and the main
result of \cite{LZ23}.

  By \cite{LZ23},  $\bbV \to T_{\Pic(X)}$ is a small resolution of the mirror
  family $V \to T_{\Pic(X)}$. So each quniversal family (see \cref{sec:quf}) gives
  a fibrewise compactification
  $\bbV \subset \bbX$. Let $\obbX \supset \bbX$ be an snc compactification.
  Recall we have the weights $w:\Sk(V)(\bbZ) \to \Pic(X)$,
  see \cref{def:skVweight}.
  The following is clear from the definition:
  
  \begin{lemma} We have a natural identification $w^{-1}(0) = \Sk(U)$.
  \end{lemma}

  By the lemma we have a unique isomorphism of free $L_{\obbV}$ modules
   $$
    A_{w^{-1}(0) \subset \Sk(V),L_{\obbX}} =
    A_{\Sk(U),L_X} \otimes_{L_X} L_{\obbV}
    $$

    identifying the bases via $w^{-1}(0) = \Sk(U)(\bbZ)$.

    \begin{proposition} \label{prop:ext} The isomorphism of free modules immediately
      above is an isomorphism of $L_{\obbV}$-algebras.
    \end{proposition}
    \begin{proof} This follows immediately from the basic convexity
      \cite[20.1]{KY22} and the main result of \cite{LZ23}, identifying
      $V \to T_{\Pic(X)}$ with the universal family \cite[4.24]{MLPv1}.
    \end{proof}

    We note by the main result of \cite{LZ23}, that the GHKI family
    (restricted to $T_{\Pic(X)}$) 
    is modular, the universal family of \cite[4.24]{MLPv1}.

        Let $W$ be an effective ample Weil
    divisor with support $E$ (see \cite[4.20]{MLPv1}). There is a corresponding monomial in
    the $\theta_{Z \in \Sk(V)}$, homogeneous of $\Pic(X)$ weight $W$,
    which we indicate by $\pi$. 
    $\pi$ is in the $W$ weightspace of $A_{C}$, and by \cref{prop:action}
    generates
    the ideal of the boundary and the localisation at $\pi$
    is $A_{\Sk(V)}$. So we have:

    \begin{corollary} \label{cor:LZ}
      $$
      (A_{w^{-1}(\bbN \cdot W) \cap C,L_{\obbX}})_{(\pi)} =
      A_{\Sk(U),L_{X}} \otimes_{L_{X}} L_{\obbX}.
      $$
    \end{corollary}

    Note $A_{w^{-1}(\bbN \cdot W)}$ is, by definition, the coordinate ring
      of the GIT quotient $A_C//W$. Thus by \cref{prop:action} the
      localization  $(A_{w^{-1}(\bbN \cdot W) \cap C,L_{\obbX}})_{(\pi)}$
      is the coordinate ring
      $$
      \cO(A_{C}//W \setminus \partial) = (A_{w^{-1}(\bbN \cdot W) \cap C,L_{\obbX}})_{(\pi)}.
      $$

      We have:
      
      \begin{corollary} \label{cor:versal} For any snc compactification $\bbX \subset \obbX$,
        $(A_C//W \setminus \partial) \to T_{\Pic(\obbX)}$ is a versal
        deformation of $U \coloneqq X \setminus E$.
      \end{corollary}
      \begin{proof} Immediate from \cref{cor:LZ} and the modularity of
        $V \to T_{\Pic(X)}$.
        \end{proof} 

        \begin{definition} \label{def:gen}
          By the {\it generic locus} in the mirror family
        $A_{C,L_{\obbV}}$ we mean the locus where the fibres of
          $(A_C//W \setminus \partial) \to T_{\Pic(\obbX)}$ are smooth.
          \end{definition}

 Next we apply \cref{sec:GITbasics} to $A_C$ with its $\Pic(X)$ grading.

\section{VGIT and toric models for $A_C$} 

        \begin{remark} \label{rem:inductions} 
        In studying $A_C$, and its GIT quotients, we have one kind of
        induction we can use, induction
        on the charge. But we can also consider the quotient
        $A_C/I_Z$ for irreducible component $Z \subset E$, here there is
        a reduction in dimension: By 
        basic convexity \cite[24.7]{KY22}, we can compute the structure constants for
        this quotient using disks contained in a fixed general fibre
        of $\theta_{Z \in \Sk(G)}$ (roughly $A_C/I_Z$ is mirror to a general
        fibre of $\theta_{Z \in \Sk(G)}$)
        with further restrictions as in \cref{prop:bbY},
        when we restrict the weights. For example 
        $A_{C,p^*(\Nef(\oX)),L_{\obbV}}/I_Z$ depends only on a generic fibre
        of $\theta_{Z \in \Sk(G)}$ on $\bbY$ in \cref{prop:bbY}. 
        On this generic fibre the deformation is isomorphic
        to the toric case, indeed by the very construction of the deformation
        in \cite{GHKIv2}, which gives 
        local toric charts (covering all but the $0$ stratum of the
        vertex). We apply this in a simple way next:
      \end{remark}

        \begin{lemma} \label{lem:P1}
          Let $L \in \Nef(X)$. For each irreducible $Z \subset E$
          $(A_C/I_Z)//L$ is a trivial family of $\bbP^1$s.
        \end{lemma}
        \begin{proof} $(A_C/I_Z)//L$ is purely toric, as noted in \cref{rem:inductions}. The result is easy in the toric case. \end{proof}

        \begin{proposition} \label{prop:lpsings}
          Let $p: X \to \oX$ be the blowdown of a set
          of disjoint internal $-1$ curves, and $L \coloneqq p^*(A)$ for
          $A \in \Ample(\oX)$. Then for each fibre $(X',E')$ of
          $(A_C//W,\partial) \to T_{\Pic(\obbX)}$ (for any snc compactification
          $\bbX \subset \obbX$) the following hold:
          \begin{enumerate}
          \item $K_{X'} + E'$ is trivial and log canonical.
          \item $(X',E')$ is a Looijenga pair with at worst du Val
            singularities off $E'$.
          \end{enumerate}
           \end{proposition}

\begin{proof} By \cref{prop:bbY} and \cref{prop:eqdef} we can assume
  $p$ is the identity, i.e. $L \in \Ample(X)$. 

    The
    set of fibres is independent of compactification, see \cref{rem:coeff}, so
    we drop $L_{\obbX}$ from the notation, and are free to use whatever compactification
    we like. As in the proof of \cref{prop:eqdef}, the coordinate ring
    $A_{C_L}$ can be computed from the formal completion of
    $\ocY \to \bbA^1_L$ at 
    $0 \in \VVert$, notation as in \cref{prop:dvres}. 
     The set of fibres of the mirror family do not change if we replace $L$ be
    a power, and so we can assume $\ocY$ has a semi-stable resolution
    $\cY \to \ocY$. Let $t \in \bbA^1_L$ be a monomial parameter. The
    basis $C_L = \ocY_L^{\beth} \cap \Sk(V)$ of the coordinate ring
    is naturally a cone, with
    $h \coloneqq t^{trop}:C_L \to \bbR_{\geq0}$, which gives the degree of
    the $\bbN$-graded $A_{C_L}$. $C_L$ is a cone over its height one
    part $B$, which is triangulated, (we indicate this by $\uB$)
    the essential (meaning corresponding
    to poles of the volume form) dual complex of the
    central fibre. $\uB$ has natural boundary, $\partial \uB$, given by the vanishing of
    some $\theta_{Z \in \Sk(G)}^{\trop}$, $\partial \uB$ is a triangulated
    circle, the dual complex of
    the vertex central fibre $\VVert(X,E)$ of \cref{lem:center}.

    \begin{claim} $\uB$ satisfies the conditions of \cite[21.1]{KY22}. \end{claim}
    \begin{proof} For the boundary simplices, $\partial \uB$ this is clear (its
      a triangulated circle). For interior simplices -- which
      correspond to valuations with center at $0 \in \VVert$, the result follows from
      \cite[Th 2]{KX} (applied to a resolution). 
    \end{proof}

    The boundary family $\partial(A_{C,L_X}//L)$ is purely toric,
    a family of cycles of rational curves, see \cref{rem:inductions}.
    Now the fibres $K_{X' + E'}$ are
    slc, trivial and log canonical by 
    \cite[21.1]{KY22} {\it pushing} fibres to the central
    fibre just as in the proof of \cite[22.1]{KY22}. The fibres are normal by
    \cite[22.1]{KY22}. $X'$ is du Val off $E'$ by \cite[7.2]{HKY20}.
    This completes the proof.  \end{proof}

        \begin{assumptions} \label{ass:bd} We fix $p_i: X \to X_i$ the blowdown of
          interior $-1$ curves $Q_i$ meeting $E_i$, $i = 1,2$. We
          assume $E_1 \neq E_2$. Let $q: X \to \oX$ be the composition
          (blow down both $Q_i$). Note for proving \cref{thm:modular} we
          can always assume this by induction on the charge,
          \cref{prop:tbred} and \cref{lem:ch-1}. 
          \end{assumptions}

          \begin{notation} \label{nota:sigma} For $L \in \Pic(X)$ let
            $\sigma(L)$ be the minimal cone in the GIT fan containing $L$
            (i.e. the unique fan containing $L$ in its (relative) interior).
          \end{notation}

          \begin{proposition} \label{prop:amps} Let $p: X \to \oX$
            be the blowdown of an internal $-1$ curve, meeting
            $E$ along the irreducible component $Z \subset E$. The
            following hold:
            \begin{enumerate}
              \item Let $L_1,L_2 \in \Ample(X)$. 
            The
            semi-stable loci have the same intersection with the boundary,
            see \cref{const:boundary}. Moreover this intersection
            consists entirely of stable points (for either linearization).
            \item $L_1,L_2$ as in (1). The birational map
              $A_C//L_1 \dasharrow A_C//L_2$ is an isomorphism in a neighborhood
              of the boundary (on domain and target). 
          \item Let $L \in \Ample(X)$, and $L_1 \in p^*(\Ample(\oX))$.
            Then the exceptional locus of
            $A_C//L \dasharrow A_C//L_1$ is disjoint from all boundary divisors
            other than the one associated to $Z$.

            \end{enumerate}
            
          \end{proposition}
          \begin{proof} We consider (1). By \cref{prop:bbY}
           $A_C//L_i$ is 
        computed using the formal completion of $\ocY \to \TV(\Nef(X))$ at
        the $0$-stratum of the vertex central fibre. 
        Moreover, by \cref{const:boundary}, for $(A_C/I_Z)//L_i$ we can further restrict to a generic
        fibre of $\Theta_{Z \in G}$ -- which we note, in the vertex, is
        a transverse slice of the corresponding $1$-stratum. 
        The formal deformation here is purely toric by \cite{GHKIv2}, isomorphic
        to the analogous construction for $(X,E)$ toric. Indeed, this is
        clear from the construction in 
        \cite{GHKIv2}, which gives 
        local toric charts (covering all but the $0$ stratum of the
        vertex). So we are reduced to the toric case, which is easy to check, 
        e.g. the equality $A_C = \Cox(X)$ 
        is  part of the theory of the secondary fan. See \cite[Chapter 14]{CLS}. In the toric
        case ample linearizations have no strictly semi-stable points, and
        the semi-stable ($=$ stable) locus is the same for all amples.
        the stable loci are equal for all amples.This gives
        (1). (2) follows immediately from (1).

        (3) follows from the same reasoning: By the argument for
        \cref{prop:bbY} the rational map
        is determined by the formal completion
         of the GHKI mirror along the
         the boundary $1$-stratum $S_p$. The family here is not toric,
         there is a single {\it moving worm}, see \cref{rem:mw}, along the
         one stratum of $\VVert$ corresponding to $Z$. But, as in the previous
         paragraph, the exceptional locus of the rational map near the boundary
         corresponding to $W\neq Z \subset E$ is determined by further
         restricting to a general fibre of $\theta_{W \in \Sk(G)}$, and here
         the deformation is toric (e.g. by the local charts of \cite{GHKIv2}).
         So the result follows as in (1).
         
        This completes the proof. 
        
      \end{proof}

        \begin{corollary} \label{cor:regular} Pick $L \in \Ample(X)$,
          $L_1 \in p_1^*(\Ample(X_1))$. Then the birational
          map $A_C//L \dasharrow A_C//L_1$ is regular in a neighborhood
          of the boundary. 
          \end{corollary} 
          \begin{proof} By \cref{prop:amps}, the question is independent
            of $L$ (in $\Ample(X)$), so we can choose it so that the GIT
            chamber for $L_1$ is in the boundary of the GIT chamber for
            $L_2$. Then the birational map is regular by 
            the general GIT theory. \todo{reference?} This completes
            the proof. \end{proof} 

          \begin{remark} \label{rem:formalamps} \cref{prop:amps} , \cref{cor:regular}
            and \cref{lem:P1} hold as well
            for the mirror algebra associated to the formal deformations as
            in \cref{rem:beinduct}.
        Indeed the proof  only involves the formal deformation.
        \end{remark}

Now we assume we have $p:X \to X_{1,2}$ blowdown of disjoint internal
$-1$ curves, $Q_1,Q_2$, with $Q_i$ meeting $E_i$, and $E_1 \neq E_2$. 
Let $p_i: X \to X_i$ blowdown $Q_j$ ($i \neq j$). 
Let $L \in \Ample(X)$, $L_i \in \bogus(p_i)$, $L_{1,2} \in \bogus(p_{1,2})$.

          We have a commutative diagram of GIT
          quotients

\begin{equation} \label{eq:cd}
\begin{tikzcd}
		A_C//L  \rar[dotted]{p_1} \dar[dotted]{p_2} & A_C//L_1 \dar[dotted]{q_1}\\
		A_C//L_2  \rar[dotted]{q_2} & A_C//L_{1,2}
\end{tikzcd}
\end{equation}

          \begin{proposition} \label{prop:induct} Assumptions as
            immediately above. Referring to the commutative diagram
            \cref{eq:cd}: 
            After restricting to a fibre $p \in T_{\oV}$
            (which we leave out of the notation, thus we here abuse notation
            and write $A_C$ rather than $A_{C,L_{\oV},p}$) 
            The following hold: 
            \begin{enumerate}
            \item Each map is birational, regular in a neighborhood of
              the boundary, and carries the boundary isomorphically onto
              its image.

            \item For any ample $M$, each fibre of $A_C//M$ is smooth
              in a neighborhood of the boundary.

            \item For any fibre (over a point of $L_{\obbV}$),
              and any  $M_1,M_2 \in
              \Ample(X)$, the rational map $A_C//M_1 \dasharrow A_C//M_2$
              induces an isomorphism of the minimal resolutions

            \item The restriction $A_C//M \to L_{\obbV}$ to the generic
              locus, see \cref{def:gen},
              is smooth, for any ample $M$, and the natural birational
              map between any two restrictions (for two ample $M_i$) is an
              isomorphism.

            \item If we restrict to a fibre over the generic locus, then on
              the minimal resolution (which we indicate with a tilde) 
              there are disjoint $-1$ curves $Q_1, Q_2 \subset \widetilde{A_C//L}$
              such that $p_i$ is the blowdown of $Q_j$, and $q_i$ the blowdown
              of $Q_i$.
            \end{enumerate}
      \end{proposition}
      \begin{proof}
        For (1): If we choose $L$ sufficiently close to the face
        $p_1^*(\Nef(X_1)) \subset \Nef(X)$ then $p_1$ is regular by VGIT.
        Now near the boundary we have regularity for all $L$ by
        \cref{cor:regular}. Note: this same reasoning holds in the
        {\it formal case} discussed in \cref{rem:beinduct}.  The same
        reasoning applies to the other maps.  This gives gives regularity
        in (1). Next we consider the boundary map in (1). By
        \cref{lem:P1}, the boundary families (in all {\it cases}, notation
        as in \cref{rem:beinduct}) are trivial, and the map takes the
        component associated to $I_E$, see ??, to the analogous component. 
        First
        consider the $q_i$. It's enough to see no boundary component
        is contracted. 
         In the $X_i$ case (notation as \cref{rem:beinduct})
         the statement holds by induction, now the formal case follows, as the
         statement, that no components is contracted  is deformation
         invariant (since the boundary family is a family of smooth
         $\bbP^2s$). This gives the result for the $q_i$ and this implies
         the result for the $p_i$ using commutativity of \cref{eq:cd}
         and (3) of \cref{prop:amps}. This gives (1).

         (2) implies (3) by \cref{lem:lpnb} 
            and (2) of \cref{prop:amps}.

         For the remaining statements we can restrict to the generic locus.
         Now the boundary on each fibre is ample, and so it follows from
         (2) of \cref{prop:amps} that all the maps are regular. Now using
         commutativity of the diagram and (3) of \cref{prop:amps}, its
         enough to show the $q_i$ are blowdowns of $-1$ curves. Here we use
         the notation from \cref{rem:beinduct}. For the $X_i$ case we have
         the result by induction on the charge. The formal case follows as
         what we are trying to prove is a deformation invariant statement:
         smoothness along the boundary is deformation invariant, because
         we know the boundary family is trivial so we can check it by
         computing that the intersection number of two adjacent boundary components
         is one. Now this gives smoothness of the surfaces, as the boundary
         is ample on the minimal resolution (the surfaces are generic).
         To see $q_i$ (over any fibre of the base torus for the formal case)
         are blowdowns of a single $-1$ curve: We have a family of regular
         birational maps between smooth Looijenga pairs, carrying the boundary
         isomorphically to its image. The relative Picard number is constant,
         namely one so the formal case follows by deformation from the $X_i$ case.
         This gives (5). 

         For (4) we already have that any two families are smooth, and the
         birational map is an isomorphism near the boundary. But the boundary
         is ample, it follows the birational map, or its inverse, has no
         exceptional locus and thus (as its between smooth families) is an iso.
         This completes the proof.

\end{proof} 
      \begin{corollary} \label{cor:tm} For any fibre over the generic locus 
         $\widetilde{A_C//L}  \dasharrow \widetilde{A_C//L_{1,2}}$  has the
         same deformation type 
         as $X \to X_{1,2}$ -- where as above, the tilde means the minimal
         resolution away from the boundary.

         For any ample $M$, any fibre $(A_C//W,\partial)$ over the generic
         locus has the same deformation type as $(X,E)$. 
      \end{corollary}

      \begin{proof} The first paragraph follows from \cref{thm:modular} 
        and \cref{prop:induct}, by induction on the charge.
        
        The second follows from the first, and induction on charge, using
        (3) of \cref{prop:induct}. 
      \end{proof}
      
      \begin{lemma} \label{lem:lpnb} Let $f:(X_1,E_1) \dasharrow (X_2,E_2)$
        be a birational rational map of smooth Looijenga pairs, which
        gives a regular isomorphism of some neighborhood of $E_1$ to some
        neighborhood of $E_2$, and takes $E_1$ isomorphically to $E_2$.
        Then $f$ is an isomorphism.
      \end{lemma}
      \begin{proof} It's enough to show $f$ has no exceptional divisor (as we
        can apply the same argument to $f^{-1}$. So consider $V \subset X_1$
        the regular locus, note the complement $X_1 \ V$ is zero dimensional,
        so $\Pic(V) = \Pic(X_1)$.
        $$
        (f|_V)^*(K_{X_2}) \equiv (f|_V)^*(-E_2) = - E_1 \equiv K_{X_1}.
        $$
        
        It follows that $K_{X_1} \equiv f^*(K_{X_2}$ which implies (by the
        derivative sequence) $f$ has no exceptional divisors. This
        completes the proof.
      \end{proof}

\section{More on the period point}  \label{sec:mpp}  
Now we can deduce modularity for the families of pairs
     $(A_C//L,\partial)$: 

      By \cite{LZ23}, there is a canonical (mirror symmetric)
      compactification of the GHKI family to a  modular family -- the
      universal
      deformation \cite[4.24]{MLPv1} of $(X,E)$. Indicate this as $(\cX,\cE) \to T_{\Pic(X)}$.

      \begin{proposition} \label{prop:LZ} Notation as in \cref{cor:LZ}.
     
        The base extension
        $(\cX,\cE) \times_{T_{\Pic(X)}} T_{\Pic(\obbX)}$ is
      $$
      (A_{C,L_{\obbX}}//W,\partial) \to T_{\obbX}.
      $$
    \end{proposition}
    \begin{proof} By \cref{cor:LZ} the analogous statement holds for
      the complement of the boundaries (e.g. $\cE \subset \cX$). So we
      just need to know the fibrewise compactifications have the same
      deformation type, and this follows from \cref{cor:tm}.
    \end{proof}

    We indicate the family from \cref{prop:LZ} by
    $(\cX,\cE) \to T_{\obbX}$.

    \begin{theorem} \label{thm:mpp}  For each GIT general $L \in \Ample(X)$.
      The VGIT map
      $$
      (A_{C,L_{\obbX}}//L,\partial) \dasharrow
      (A_{C,L_{\obbX}}//W,\partial)
      $$
      gives a small resolution.

      Choose any splitting $T_{\Pic(X)} \subset T_{\obbX}$. The restriction
      $$(
      A_{C,L_{\obbX}}//L,\partial)|_{T_{\Pic(X)}}  \dasharrow
      (A_{C,L_{\obbX}}//W,\partial) = \cX
      $$
      is the small resolution $\bbX \dasharrow \cX$ for one of
      the quniversal families. 
    \end{theorem}

    \begin{proof} By ??, for some marking of the boundary, the marked period
      point for $(\cX,\cE) \to T_{\obbX}$ is the natural map
      $T_{\obbX} \to T_{\bbX}= T_{X}$
      (induced by the open set $\bbX \subset \obbX$).

      Consider first the restrictions of the families to 
      $T^{\sm}_{\obbX}$, the inverse image of $T^{\sm} \subset T_{\Pic(X)}$. Here
      both families are smooth. 
    By \cref{prop:induct}
    and \cref{cor:tm} the induced rational map on each fibre surface
    is an isomorphism,
    so to show the map is an isomorphism it is enough to show it is regular.
    By \cref{cor:regular} it is a regular isomorphism in a neighborhood of
    the boundary. But the boundary is ample on each fibre 
    it follows there is no exceptional locus
    for the rational map in either direction, and thus the map is an isomorphism.
    
    Now consider the full families. By the $T^{\sm}$ case, and \cref{prop:amps}
    the birational map is a small modification. So
    we just need to show  $(A_{C,L_{\obbX}}//W,\partial) \to T_{\Pic(\obbX)}$ is
    smooth. For this its enough to consider the restriction to any splitting
    $T_{\Pic(X)} \subset T_{\Pic(\obbX)}$,
    as the full family is iso to a base extension -- see the construction in
    \cref{sec:cm}. Each quniversal family gives a small resolution
    $\bbX \to \cX$ (the base of all families is now $T_{\Pic(X)}$. And by
    \cite[4.14]{MLPv1}, each small $\bbQ$-factorial modification of $\bbX$
    is another quniversal family. So it is enough to show
    $(A_{C,L_{\obbX}}//L)|_{T_{\Pic(X)}}$ has the full $\bbQ$-Picard group of
    $\bbX$, namely $\Pic(X)$. But each $\bbQ$-line bundle on the target
    of
    $$
    (A_{C,L_{\obbX}}//L)|_{T_{\Pic(X)}} \dasharrow \bbX
    $$
    pullbacks back
    to a $\bbQ$-line bundle on the domain, namely the line bundle on the GIT
    quotient corresponding to this weight, and this pullback is injective,
    because we can further restrict to the generic locus (see \cref{def:gen}  ) $T_{\Pic(X)}^{\sm}$, $\Pic(\bbX) = \Pic(X)$
    under restriction to any smooth fibre. This completes the proof.
    \end{proof}

      \begin{corollary} Let $T_{X} \subset T_{\obbX}$ be a splitting. Take
        $A \in \Ample(X)$ GIT general. There is a marking of
        the boundary of $(A_{C,L_{\obbX}}//A,\partial)|_{T_X} \to T_X$
        so that the marked period point is the identity (where $\Pic$
        is marked via \cref{thm:mpp}. This is isomorphic to the quniversal
        family associated to $A$, see \cref{thm:mpp}.
      \end{corollary}

  \begin{remark} \label{rem:neg} The theory of the quniversal families in
    \cite{MLPv1} does not require $U$ to be affine. The domain of the
    rational map in \cref{thm:mpp} also makes sense in this generality --
    as noted in \cref{rem:small} it only depends on the formal GHKI family.
    We 
    expect the theorem to hold, with the same proof, in general. It is natural
    to hope this could be
    used to give a new proof of the main result of \cite{HPV24} -- one
    that does not use  the
    log Kulikov families and results of Friedman-Morgan (both special to
    dimension two). 
    Alternatively we expect \cref{thm:mpp} could be deduced (in full generality)
    from  the methods of \cite{HPV24}.
    \end{remark} 

\begin{proposition} \label{prop:sss} The following hold:
    \begin{enumerate}
    \item For any $L \in \Ample(X)$ the strictly semi-stable locus
      $A_{C,L_X}^{\msss} \subset A_{C,L,X}^{\mss}$ has codimension at least two.
      \item  For all $L_1,L_2 \in \Ample(X)$,
    The $A_{C,L_X}^{\mss}(L_i)$ are equal outside codimension two (i.e. there
    are codimension two closed subsets of each whose complements are equal).
    The same holds for the stable loci.
    \end{enumerate}
  \end{proposition}
  \begin{proof} Consider (1). The strictly semi-stable locus is the inverse
    image of a closed subset of the GIT quotient \todo{give reference or
      a short proof}, and by \cref{prop:amps}, the inverse image of a closed
    set disjoint from the boundary of this Looijenga pair. Over $T^{\sm}$
    the boundary is ample, so this is codimension two. (1) follows. For (2),
    by (1), we need only consider the stable loci. Again the result follows
    from \cref{prop:amps} and positivity of the boundary. This completes
    the proof. \end{proof}

  \begin{proposition} \label{prop:fut} For $L \in \Ample(X)$ GIT general,
    and consider the quotient 
     $ q: A_{C,X}^{\mss}(L)  \rightarrow    A_{C,X}^{\mss}(L)/T_{A_1(X,\bbZ)} = \bbX$
     (last equality by \cref{thm:mpp}. The
     linearized line bundle given by $M \in \Pic(X)$ descends to the
     line bundle $M \in \Pic(\bbX) = \Pic(X)$.

     $T_{A_1(X,\bbZ)}$ acts freely and properly on $A_{C,X}^{\mss}(L)$ and
     the quotient $q$ is the universal torsor of $\bbX$. 
  \end{proposition}
  \begin{proof}
    Suppose we have the first paragraph. Then the action is free and
    proper by (the easy direction of) the Kempf descent lemma \todo{reference}. 
    $T_{A_1(X,\bbZ)}$ principal bundles over any base $W$ are classified by
    the descent homomorphism $\Pic(X) \to \Pic(W)$. It follows $q$ is the
    universal torsor.

    So it is enough to prove the first paragraph. Equivalently,
    that the pullback of $M \in \Pic(\bbX) = \Pic(X)$ to the semi-stable locus
    is the linearized line bundle given by this character. For this
    we can remove a codimension two closed subset (because the mirror
    algebra is normal), and so by \cref{prop:sss} we can, for given $M$
    choose whichever ample $L$ we like.

    Let $b: X \to \oX$ be a toric model, and $M' \in \Ample(\oX)$
    general. If $\Nef(\oX)$ is in the boundary of the GIT
          chamber for $L$, we
          have an associated regular map of GIT quotients.

          For the linearization $M \coloneqq p^*(M')$, we have the basic induction
          statement on the deformation, and note the one parameter GHKI
          family for $(\oX,E,M')$ is toric, see \cite[3.40]{GHKIv2}. So the
          mirror algebra given by the formal case in \cref{rem:beinduct} is a purely toric
          construction, the mirror algebra is the monoid algebra (with appropriate
          coefficients) for a toric monoid. It follows from \cref{prop:bbY}
          that $A_{C_L}$ is the monoid algebra (for its coefficients)
          for the same monoid, and this is (a base extension of) the
          coordinate ring for $(\oX,M')$. In particular on the GIT quotient
          $\cO(1)$ is invertible, and so by \cref{lem:descend} the linearized
          line bundle on $A_C^{\mss}(M)$ descends. 
          Since
          $A_C^{\mss}(L) \subset A_C^{\mss}(M)$, the same holds for this linearized
          bundle on $A_C^{\mss}(L)$. But now varying the toric model,
          such line bundles generate
          $\Pic(X)$, by \cref{lem:gen}. This completes the proof. 
          \end{proof}

  \section{$A_{C,L_X} = \Cox(\bbX)$} \label{sec:cm}

  \begin{assumption} \label{ass:cox} Here we assume $X$ has
    enough toric models as in \cref{lem:gen} and a pair of disjoint internal $-1$ curves
    as in \cref{lem:ch-1}
          Note we can assume this for
          the proof of \cref{thm:modular} by \cref{prop:tbred}. 
        \end{assumption}

      \begin{lemma} \label{lem:descend} Let $A$ be a finitely generated $P$-graded domain,
        for a finite rank free Abelian group $P$ and consider the linearization
        of the algebraic torus $T \coloneqq \Hom(P,\bbG_m)$ on
        the trivial line bundle over $\Spec(A)$ determined by the character
        $w \in P$ (i.e. the setup of \cref{sec:GITbasics}). Let
        $R \coloneqq \oplus_{n \geq 0} A_{n \cdot w} \subset A$ be the ring of
        invariants. The following are equivalent:
        \begin{enumerate}
          \item The sheaf $\cO(1)$ on the GIT quotient
            $A//w \coloneqq \Proj(R)$ is locally free.
          \item The 
stabilizer of a point in the semi-stable locus
$A^{\mss}(w)$ is in the kernel of $w$, and
\item The Linearized line bundle associated to $w$ descends to the
  GIT quotient (ie the restriction to $A^{\mss}(w)$ is isomorphic as
  linearized line bundle to the pullback of a linebundle from the
  quotient (with its canonical linearization). 
  \end{enumerate} 
\end{lemma}
\begin{proof}
  We let $L$ indicate the linearized line bundle associated to $w$. If
  $L$ descends, it descends to $\cO(1)$ so (3) implies (1). (3) implies (2)
  by the definition of the action on the linearized bundle, and on the
  pullback.
  (2) implies (3) by Kempf descent. So it's enough to show (1) implies (2):

  Note for each homogeneous $f \in R \subset A$,
  $(A_f)_0 = (R_f)_0$ (where these indicate the subring of degree $0$ elements)
  and $\Proj(R)$ is covered by $\Spec((R_f)_0)$, and $A^{\mss}(w)$ by
  $\Spec(A_f)$, over homogeneous $f \in R$ of positive degree. So
  replacing $A$ by various $A_f$ we can assume the graded $R$ module
  $R(1)$ is free, generated by $f \in A_w$. In particular
  $A^{\mss}(w) =\Spec(A_f) \subset \Spec(A)$. But now suppose
  $\lambda x = x$ for $x \in A^{\mss}(w)$. Then
  $$
  f(x) = f(\lambda x) = (\lambda f)(x) = (w(\lambda)f)(x) =
  w(\lambda) f(x)
  $$
  and so $\lambda \in \ker(w)$ as required.
  \end{proof} 
  Let $\bbT \to \bbX$ be the universal torsor, and 
  let $\bbG \to \bbU \coloneqq \bbX \setminus \bbE$ be its restriction.

        \begin{corollary} \label{cor:dm} $A_{\Sk(V),L_X} = \cO(\bbG)$.
                  \end{corollary}
        \begin{proof} Since by assumption $U$ is affine,
          there is an effective ample Weil divisor with support
          $E$. Then for the associated linearization, the corresponding
          monomial, $\pi$, in the $\theta_{Z \in \Sk(V)}$ gives an invariant section of
          this weight. The non-vanishing locus of this monomial,
          in $A_{C,L_X}^{\mss}(L) = \bbT$ (equality by \cref{prop:fut}) is
          $\bbG \subset \bbT$. 
          The semi-stable locus is the same, up to
          codimension two, for all amples, by \cref{prop:sss}.
          It follows that, up to codimension two, this non-vanishing
          locus in $A_{C,L_X}^{\mss}(L)$ is $\Spec((A_{C,L_X})_{\pi})$, and
          by \cref{prop:action} the localisation $(A_{C,L_X})_{\pi} = A_{\Sk(V)}$.
          Removing codimension two does not change the regular functions (on
          these normal spaces). The result follows. 
          \end{proof} 

          Now we follow the general philosophy of the three basic conjectures
          to prove $A_{C,\bbX} = \Cox(\bbX)$ -- a version of the double mirror
          conjecture is \cref{cor:dm}.

          Note each irreducible $Z \subset E$ corresponds to a component
          $\bbZ \subset \bbE \subset \bbX$ and then its inverse image
          in $\bbT \to \bbX$ and so a boundary divisor to $\bbG$, which we
          indicate $Z \in \Sk(\bbG)$.

          The following is a special case of the main results of \cite{CMMM}:
          \begin{proposition} \label{prop:ndsym} Notation as immediately above.
            The following hold:
            \begin{enumerate}
            \item (non-degeneracy) For $0 \neq \alpha_i \in L_X$ and $v_i \in
              \Sk(V)$, and $Z \subset E$ we have
              $$
              (\sum \alpha_i \theta_{v_i \in \Sk(V)})^{\trop}(Z \in \bbG)
              = \min_i \theta_{v_i}^{\trop}(Z \in \bbG).$$
            \item (symmetry)
              $$ \theta_{v \in \Sk(V)}^{\trop}(Z \in \Sk(\bbG)) =
              \theta_{Z \in \Sk(G)}^{\trop}(v).
              $$
              \end{enumerate}
            \end{proposition}
            \begin{proof} The fibres $\bbG_t \to \bbX_t$, $t \in T^{\sm}$ and
              $V$ are Fock-Goncharov dual cluster varieties -- the second of $\cX$-type,
              the first associated $\cA_t$ cluster varieties. For these particular
              cluster varieties \cite{CMMM} proves finiteness of
              broken lines, i.e. $\Theta = \cX^{\trop}$ in the notation of \cite{GHKK}.
              This, together with \cref{cor:dm} gives the Full Fock Goncharov
              conjecture in these cases (see \cite[0.6]{GHKK}). Now the symmetry
              and non-degeneracy results of \cite{CMMM} (which are general statements
              for any pair of Fock-Goncharov dual cluster varieties) immediately
              imply (1-2). \end{proof}

            \begin{corollary} \label{cox:cor} We have a canonical isomorphism of $\Pic(X)$ graded
              rings $A_{C,L_X} = \Cox(\bbX)$. \end{corollary}
            \begin{proof} This follows from the general philosophy of the three
              conjectures in \cref{rem:conjs}, using \cref{cor:dm} and \cref{prop:ndsym}. See \cite[24.6]{KY22}. 
              \end{proof}

          \subsection{The analogs of $1_Z \in \Cox(X)$ in $A_C$}
The irreducible components $Z \subset E$ give invertible functions on $G$:
    any section of any line bundle on $X$ gives a function on
    the universal torsor $\cT$; here
    we take $ 1_Z \in H^0(X,\cO(Z))$.  
    Together this gives a map 
    $\cT \to T^E$, where the target is the split torus with character
  lattice $\bbZ_E$, free Abelian with basis the irreducible components,
  which we can view as 
  Weil divisors supported on $E$. We write $\bbZ^E$ for the dual.

  As explained in \cref{sec:ta}, by the basic convexity statement
\cite[20.1]{KY22}, \cite[15.6]{KY20}, this
gives a $\bbZ^E$-grading on the mirror    algebra $A_{\Sk(V)}$ or equivalently induces an equivariant action of the dual torus $T_E$ on
$V \to T_{\Pic(X)}$, where $T_E \to T_{\Pic}$ is induced by $\bbZ_E \to \Pic$.

  This gives, by the general construction of \cite{LoganThesis}, a translation action of
  $\bbZ_{\bbR}$ on $\Sk(V)$, and in particular a copy
  $\bbZ_E = \bbZ_E + 0 \subset \Sk(V)(\bbZ)$. Let 
  $Z \in \bbZ_E \subset \Sk(V)(\bbZ)$ indicate the corresponding basis element
  of $\bbZ_E$ (for each irreducible $Z \subset E$). 

\begin{notation} \label{nota:Z>0} We write e.g. $Z > 0 \subset \Sk(V)$ for
    the locus $\theta^{\trop}_{Z \in \Sk(G)} > 0$.  As noted in \cref{const:boundary}
    $A_{Z \geq 0} \subset A_{\Sk(V)}$ is a subalgebra. We will sometimes write
    $[Z > 0] \subset A_{Z \geq 0}$ for the corresponding ideal -- the ideal
    with basis $[Z > 0] \subset \Sk(V)$. Note
    $A_C/([Z > 0] \cap C)$ has basis $Z^{\perp} \cap C$.
        \end{notation}

  \begin{proposition} \label{prop:action} Notation as
    immediately above. The following hold:
      \begin{enumerate}
      \item For $e \in \bbZ_E$, $\theta_e \cdot \theta_v = \theta_{e + v}$.
      \item For $f$ a $T_E$ eigenfunction of weight $w$,
        $f^{\trop}(v + e) = f^{\trop}(v) + w(e)$.
      \item For $W,Z \subset E$ irreducible components,
        $\theta_{W \in \Sk(G)}^{\trop}(Z \in \Sk(V)) =\delta_{Z,W}$
      \item The ideal $[Z >0] \subset A_C$ is generated by
        $\Theta_{Z \in \Sk(V)}$. 
      \end{enumerate}

      Moreover: The localisation of $A_C$ at the product of all
      $\theta_{Z \in \Sk(V)}$ is $A_{\Sk(V)(\bbZ)}$.
          \end{proposition}
    \begin{proof} (1-2) are general results, hold for any torus
      action on any log CY with max boundary, see \cite{LoganThesis}. (3) is
      an instance of this: $\theta_{W \in \Sk(G)}$ is an eigenfunction
      for the $T_E$ action on $A_{\Sk(G)}$, with weight given by
      the various $(1_Z)^{\trop}(W) = \delta_{Z,W}$. Consider (4).
      We have by (2-3)
      $$
      \theta_{W \in \Sk(G)}^{\trop}(v + Z) = \theta_{W \in \Sk(G)}^{\trop}(v)
      + \delta_{W,Z}.
      $$
      So if $v \in I_Z \coloneqq [Z > 0]$, $v + (-Z) \in C$, and
      $\theta_{v \in \Sk(V)} = \theta_{v + (-Z)} \cdot \theta_{Z \in \Sk(V)}$.
      This gives (4).

      For the final statement (about localisation): by (1) all the theta
      functions from $\bbZ \subset \Sk(V)(\bbZ)$ are invertible in
      $A_{\Sk(V)}$, so the localisation is a subring of $A_{\Sk(V)}$.
      $A_{\Sk(V)}$ is contained in the localisation by (2-3). This
      completes the proof.

      \end{proof}

   \subsection{$\ADM$ action} 

Different markings of $\Pic$ are related by the group
$\ADM$ as in \cite[5.15]{MLP}. By \cite{LZ23}, each quniversal family
$\bbX$ gives a fibrewise snc compactification of a small resolution of
$V \to T_{\Pic(X)}$. We have a modular action of $\ADM$ on $\bbX$, regular in
codimension one, see \cite[?.?]{MLP}, giving an equivariant action on
$V \to T_{\Pic(X)}$. In particular we
get an equivariant action on $\Sk(V) \to \Pic(X)$. 
We can extend the action on $\bbX$ to an action on 
an snc compactification $\bbX \subset \obbX$, regular in codimension one. This
gives an action of $\ADM$ on the mirror algebra $A_{\Sk(V),L_{\obbX}}$ (permuting
theta functions by the action on $\Sk(V)$
and curve class by the action on $\obbX$). 
$E \subset \partial \bbX$ are naturally
the {\it horizontal}
boundary divisor, those that dominate $T_{\Pic(X)}$, and these are fixed (individually)
by $\ADM$, thus $\ADM$ fixes the corresponding theta function, and thus
the action on $\Sk(V)$ preserves the cone $C \subset \Sk(V)$, and each of its
faces (see \cref{const:boundary}). Thus $\ADM$ preserves the subring
$A_{C,L_{\obbX}} \subset A_{\Sk(V),L_{\obbX}}$, and the weight space for any weight
$w \in \Pic(X)$ preserved by $\ADM$, e.g. the weight $W$ of \cref{cor:LZ}. Because $\ADM$
fixes each component of $E$, it also fixes the corresponding monomial $\pi$ of
\cref{cor:LZ}.
So we have a compatible action on the GIT quotient
$A_{C,L_{\obbX}}//W$, and then on both sides of the equality in \cref{cor:LZ}.

The $\ADM$ action on $A_{\Sk(V),L_{\obbX}}$ 
normalizes the
action of the boundary torus $\bbG_m^{\partial \bbX}$ and so gives an action
of $\ADM$ on the quotient, $Q \to T_{\Pic(\bbX)}$, see \cref{rem:cm},
permuting  the theta divisors (see \cref{prop:quotient}).
The above gives a compatible
action on $Q//W = (A_{C,L_{\obbX}}//W)/T_K$, notation as in the proof of
\cref{prop:quotient}. 

As noted at the start of the proof of \cref{thm:mpp}, the boundary of $(Q//W
,\partial) \to T_{\Pic(X)}$ (recall $\Pic(X) = \Pic(\bbX)$) has a marking identifying this with the universal family, see \cite[5.1]{LZ23}. The marking is determined by the values of certain theta
functions $\theta_P$ associated to $P \in w^{-1}(0)$. These $\theta$ functions
are $\ADM$ invariant, it follows that the $\ADM$ action preserves the marking.
This implies:

\begin{proposition} The action of $\ADM$ on $(Q//W = (A_{C,L_{\obbX}}//w)/T_K, \partial \to
  T_{\Pic(X)}$ is the modular action -- changing the marking of $\Pic$ and fixing the
  marking of the boundary cycle $E$.
\end{proposition}
\begin{proof} By the above for 
  $\alpha \in \ADM$ and $t \in T_{\Pic(X)}$, the map $\alpha$ from the Looijenga pair
  fibre over $t$ to the fibre over $\alpha(t)$ preserves the marking of the    boundary.
  The same is true for the modular action. Thus the two are equal, as there is
  only one such isomorphism, see \cite[4.23]{MLPv1}.
  \end{proof}

To recap:

\begin{theorem} \label{thm:invariant}  The modular action of $\ADM$ (resp. $\bbG_m^E$) on the universal
  family $(Q//w = (A_{C,L_{\obbX}}//w)/T_K, \partial \to
  T_{\Pic(X)}$ permute the theta divisors (resp. preserves each theta divisor). 
  \end{theorem}

  \section{Products and the basic conjectures} \label{sec:prod}

        Here we prove the following:

        \begin{proposition}
          Suppose the basic conjectures (double
        mirror, symmetry and independence) hold for smooth affine log CY
        $U$ with maximal boundary. Then they hold for $U \times \bbG_m$.
      \end{proposition}
      \begin{proof} Skeleton and mirror algebra commute with product. 
        E.g. the mirror algebra of the product is the tensor product (as always,
        up to coefficients), and from this double mirror and symmetry follow
        easily. Now consider independence.
        Take
        $$
        v = (x[D],y[0]) \in \Sk(V)(\bbZ) \times \bbZ =
        \Sk(V \times \bbG_m^{\vee})(\bbZ)
        $$
        for positive integers $x,y$, where $D$ and $[0]$ indicate
        divisorial centers, $V$ is mirror to $U$, and $\bbG_m^{\vee}$ is the
        dual torus to $\bbG_m$. We will assume $x =y=1$, purely for
        notational simplicity, the same argument works in general.
        $v$ then has center 
        $$D \times \{0\} \subset X \times \bbA^1 \supset V \times \bbG_m^\vee$$
        for a partial minimal model $V \subset X$ with smooth irreducible boundary
        $D$. 
      $v$ is then a multiple of the divisorial valuation given by blowing up
      $D \times \{0\}$. The theta functions of $V \times \bbG_m^{\vee}$
      are $\theta_u t^k$ (for $t = z^{1} \in \bbZ[\Sk(\bbG_m)(\bbZ)]$). 
      We want to show that for non-zero constants $a_{u,k}$. 
        $v(\sum_u \sum_k a_{u,k} \theta_u T^k)$ is given by the minimum
        of $v(\theta_u T^k)$. We can assume 
      $v(\theta_u t^k)= v_D(\theta_u) + k \eqqcolon m$
      is constant over $u,k$. Let $y$ be a local equation for $D$. Then we
      blowup the ideal $(y,t)$ (embedding the blowup as usual in
      $V \times \bbA^1 \times \bbP^1_{Y,T}$), and consider the chart with
      $Y = 1$. So we have $t = Ty$, and $y$ cuts out the exceptional.
      Then $\theta_u = f_u y^{m-k}$ with $v_D(f_u) = 0$, and the combination is
      $$
      J \coloneqq (\sum_u \sum_k a_{u,k} f_u T^k) y^m \in R[T]
      $$
      where $R$ is the DVR for $D$. Suppose $J \in (y^{m+1})R[T]$.
      Then for each $k$,
      
      $\sum_u a_{u,k} f_u \in yR$. But then
      $$
      \sum_u a_{u,k} f_u y^{m-k}) = \sum_u a_{u,k} \theta_u \in y^{m-k+1}R
      $$
      so $$
      v_{D}(\sum_u a_{u,k} \theta_u) \geq m-k + 1 > \min v_{D}(\theta_u) = m-k
      $$
      which violates independence for $U$. 
      So $v(J) = m$. This completes the proof. 
    \end{proof}

    Now the basic three conjectures have all been proven in dimension two --
    independence and symmetry by Mandel and double mirror by Lai-Zhou. 
    \begin{remark} Here we will only need these basic conjectures in a
      very simple cases, see \cref{rem:tpic} below. \end{remark}
    
    So by the proposition we have:
    \begin{corollary} \label{cor:3} Let $U$ be smooth affine log CY surface with
      maximal boundary and trivial Picard group. Then the three basic
      conjectures hold for $G \to U$. 
    \end{corollary}
    \begin{proof} $G \to U$ is a trivial bundle since $U$ has trivial Picard group.
    \end{proof}

    \begin{remark} \label{rem:tpic} Trivial Picard group is a very strong condition on $U$. It
      is equivalent to the existence of a smooth minimal compactification with
      a toric model of relative Picard number at most two, with at most one
      exceptional divisor meeting any given boundary divisor. In these simple
      cases the mirror algebra has a simple explicit expression, and any of the basic conjectures can be checked by hand. 
            \end{remark}

      \begin{corollary} \label{cor:crtp} 
        \cref{thm:modular} holds as long as 
         $U$ has trivial Picard group. 
    \end{corollary}
    \begin{proof} In this case $G \to U$ is trivial, and so the same
      holds for the absolute mirror, $V \to T_{\Pic(X)}$, and the
      double mirror $A_{\Sk(V),L_X} = A_{\Sk(U),L_X} \otimes \bbZ[\Pic_X]$,
      $\bbV = U \times T_{\Pic(X)}$, and the $T_{A_1(X,\bbZ)}$-bundle
      $$
      \Spec(A_{\Sk(V),L_X}) = \bbG \to \bbV = T_{A_1(X,\bbZ)} \times \bbV,
      $$
      notation as \cref{cor:dm}. Symmetry and Independence give $A_{C,L_X} = \Cox(\bbX)$
      (see e.g. the proof of \cref{cox:cor} below). Note there is no moduli, in these cases:
      $\bbX = X \times T_{\Pic(X)}$;  no roots, and just a
      single (trivial) quniversal $=$  universal family. Now
      \cref{thm:modular} follows from VGIT of the
      $\Cox$ ring (in the very simple surface case), \cite{HK96}. 
      This completes the proof. 
    \end{proof}

  \section{The canonical mirror} \label{sec:canmir}
  Here we take $U$ a smooth affine log CY with  maximal boundary.
    
  Then for each compactification $U \subset \oV$ we have the mirror $L_{\oV}$ algebra
  $A_{\Sk(V), L_{\oV}}$. This depends on the compactification (though the fibres of
  the family do not, see \cref{rem:coeff}). One way to remove the
  dependence is to take the base extension
  $L_{\oV} \twoheadrightarrow \bbZ[N_1(V,\bbZ)]$, see \cite[?.?]{KY22}.
  But in doing this we can lose moduli. Here we give an alternative, in
  slightly greater generality: We begin with $U \subset V$ a partial snc
  compactification, and we attach a canonical mirror to $V$ under 
  the assumption:

  \begin{assumption} \label{ass:pic} We assume $\Pic(V)$ is free Abelian.
  \end{assumption}

  \begin{remark} If $V$ in \label{ass:pic} is projective, then the condition
    always holds: $V$ is rationally connected, which implies the vanishing
    of $H^i(V,\cO_V)$, $i > 0$, and thus equality $\Pic(V) = H^2(V,\bbZ)$.
    Then $\tor(\Pic(V)) = \tor H_1(V,\bbZ)$, which is trivial as rationally
    connected implies simply connected. \todo{can we get references? see
      kollar miyaoka mori on rationally connected}.

    In general we will have $\tor(\Pic(V)) = \tor H_1(V,\bbZ)$, and this
    can be non-trivial. But the universal cover will be a finite \'etale
    cover, itself log CY, perhaps one can extend the construction we
    give here by first
    going to this cover. 
    \end{remark}
  
    Now we have $\Pic(\oV) \twoheadrightarrow \Pic(V)$. Choose  a splitting, $s$
    (which exists by \cref{ass:pic}).
    
  We then have the base extension
  $$
  A_{\Sk(V),\bbZ[\Pic(V)^*],s} \coloneqq
  A_{\Sk(V),L_{\oV}} \otimes_{L_{\oV}} \bbZ[\Pic(V)^*].
  $$

  We let $L_V \coloneqq \bbZ[\Pic(V)^*]$. 

  \begin{proposition} \label{prop:quotient} The isomorphism class of this $L_V$-algebra, and the $\bbZ$-module basis $\theta_P z^C$, $P \in \Sk(V)$,
    $C \in \Pic(V)^*$ are independent of the splitting $s$, and the
    compactification $V \subset \oV$. 

    The $\theta$-function basis is independent of these choices up to 
    up to multiplication by (individual) characters $z^C \in L_V$,
    $C \in \Pic^*(V)$. \end{proposition}

  \begin{proof} Let $D \coloneqq \oV \setminus V$. Choose a splitting of
    $\bbZ^D \twoheadrightarrow K$, where $K \subset \Pic(\oV)$ is the kernel
    of $\Pic(\oV) \twoheadrightarrow \Pic(V)$ (we note that $\Pic(\oV)$
    is a lattice, so $K$ is as well). This now gives an equivariant action
    of $T_K$ on the mirror family over $T_{\Pic(\oV)}$, free, as the action
    on the base is free. It is clear the quotient is
    
    $\Spec(A_{\Sk(V),\bbZ[\Pic(V)^*],s})$. This shows the independence on $s$ (and
    also the splitting for $K$). Now any two splittings are related by
    a homomorphism $T_{\Pic(V)} \to T_K \subset T_{\bbZ^D}$. Now the theta
    function statement follows as the $\theta$-functions are eigenfunctions
    for the action of $T_{\bbZ^D}$. This completes the proof.
    \end{proof} 

    \begin{remark} \label{rem:cm}
      An entirely analogous construction gives a mirror for any
      snc partial minimal model $V \subset V'$ with $\Pic(V')$ a finitely generated
      Abelian group: We complete to $V' \subset \oV$, and do the same thing as above
      replacing $V$ by $V'$. The resulting algebra $A_{\Sk(V),L_{V'}}$ has canonical
      $\bbZ$-module basis, and an $L_{V'}$ module base canonical up to multiplication
      by monomials $z^C \in L_{V'}$. 
     \end{remark} 
  
     For $V$ projective, the change in the canonical (which is the same
     as the mirror from \cite{KY22}) mirror for different compactifications
     of $U$ is studied in \cite[\S 17]{KY20}. Similar results apply in
     general, for example:

     \begin{proposition} \label{prop:cq} Let $V$ satisfy \cref{ass:pic}, and let
       $V' \to V$ be proper and birational, with exceptional locus
       disjoint from $U$, so that $U \subset V'$ is itself snc. We have
       a canonical isomorphism
       $$
       A_{V',L_{V'}} \otimes_{L_{V'}} L_V = A_{V,L_V}
         $$
         preserving theta functions,
         where $L_V' \twoheadrightarrow L_V$ is induced by
         the pullback $\Pic(V) \to \Pic(V')$.
       \end{proposition}
       \begin{proof} See the proof of the second displayed isomorphism
         in \cite[17.3]{KY20}, the same argument applies.
         \end{proof} 
  
\bibliographystyle{plain}
\bibliography{dahema}

\end{document}